\documentclass[10pt]{article}

\usepackage{amsmath}    
\usepackage{amsfonts}
\usepackage{graphicx}   
\usepackage{verbatim}   
\usepackage{color}      
\usepackage{hyperref}   

\usepackage{fullpage,mathpazo}
\usepackage{amssymb}
\usepackage{amsthm}
\usepackage{bm}
\usepackage{mathtools}

\theoremstyle{thmstyleone}%
\newtheorem{theorem}{Theorem}[section]
\newtheorem{Lem}[theorem]{Lemma}
\newtheorem{Cor}[theorem]{Corollary}

\theoremstyle{thmstyletwo}%
\newtheorem{remark}[theorem]{Remark}

\theoremstyle{thmstylethree}%

\numberwithin{equation}{section}

\newcommand{\p}{\mathbb{P}}

\newcommand{\e}{\mathbb{E}}
\newcommand{\real}{\mathbb{R}}
\newcommand{\n}{\mathbb{N}}

\newcommand{\1}{{\bf 1}}

\newcommand{\rd}{\mathrm{d}}

\newcommand{\eps}{\varepsilon}
\newcommand{\pde}{\psi_{\delta,\varepsilon}}
\newcommand{\ude}{u_{\delta,\varepsilon}}
\newcommand{\sigxs}{\sigma(X_s)-\tilde{\sigma}\left(s,\tilde{X}_s\right)}
\newcommand{\sigab}{\left|\sigma(X_s)-\tilde{\sigma}(s,X_s)\right|^\alpha}

\begin{document}

\title{$L^{\alpha-1}$ distance between two one-dimensional stochastic differential equations with drift terms driven by a symmetric $\alpha$-stable process}

\author{Takuya NAKAGAWA
\footnote{
		Graduate School of Science,
		Kyoto University,
		Yoshida-honmachi, 
		Kyoto, 606--8501, 
		Japan. 
		Email: \texttt{takuya.nakagawa73@gmail.com}
  ORCiD: 0000-0003-0379-5833
}
}




\maketitle

\begin{abstract}
This paper establishes a quantitative stability theory for one-dimensional stochastic differential equations (SDEs) with non-zero drift, driven by a symmetric $\alpha$-stable process for $\alpha\in(1,2)$.
Our work generalizes the classical pathwise comparison method, pioneered by Komatsu for uniqueness problems, to address the stability of SDEs featuring both non-zero drift and, crucially, time-dependent coefficients.
We provide the first explicit convergence rates for this broad class of SDEs.
The main result is a H\"older-type estimate for the $L^{\alpha-1}(\Omega)$ distance between two solution paths, which quantifies the stability with respect to the initial values and coefficients.
A key innovation of our approach is the measurement of the distance between coefficients.
Instead of using a standard supremum norm, which would impose restrictive conditions, we introduce a weighted integral norm constructed from the transition probability density of the baseline solution.
This technique, which generalizes the framework of Nakagawa \cite{Nakagawa}, is essential for handling time-dependent perturbations and effectively localizes the error analysis.
The proof is based on a refined analysis of a mollified auxiliary function, for which we establish a new, sharper derivative estimate to control the drift terms.
Finally, we apply these stability results to derive corresponding convergence rates in probability, providing an upper bound for the tail probability of the uniform distance between solution paths.
\end{abstract}

\section{Introduction}
Let $Z=(Z_t)_{0\le t\le T}$ be a symmetric $\alpha$-stable process for $\alpha\in(1,2)$.
Its infinitesimal generator, $L_\alpha$, is defined for any Schwartz function $f \in S(\real)$ by
\begin{align*}
    L_{\alpha} f(x)&\coloneqq \int_{\real \setminus \{0\}} \left\{f(x+y)-f(x)-1_{\{|y| \le 1 \}}yf'(x) \right\} \frac{c_\alpha\rd y}{|y|^{1+\alpha}}\textrm{~~for~} x\in\real, 
\end{align*}
where the normalization constant is $c_\alpha=\pi^{-1}\Gamma(\alpha+1)\sin\left(\frac{\alpha\pi}{2}\right)$ (See, e.g., equation (3.18) in \cite{Applebaum}).
In this paper, we study the stability of solutions to the following one-dimensional SDE, which will serve as our baseline process:
\begin{align}
X_t=x_0 + \int_0^t b(X_s) \rd s +\int_0^t \sigma (X_{s-}) \rd Z_s, \label{SDE}
\end{align}
where the solution $X=(X_t)_{0\le t\le T}$ starts from $x_0 \in \real$ and evolves according to the drift coefficient $b:\real\to\real$ and the jump coefficient $\sigma:\real\to\real$.
Our primary goal is to quantitatively evaluate the distance between the solution X of SDE \eqref{SDE} and a solution $\tilde{X}$ to a more general SDE, where the drift and jump coefficients, $\tilde{b}:[0,\infty)\times \real\to \real$ and $\tilde{\sigma}:[0,\infty)\times \real\to \real$, are permitted to be time-dependent.
This framework allows us to analyze the stability of the process under a much broader class of perturbations.

To place our problem in context, it is helpful to first recall the classical case of Brownian motion $(\alpha=2)$.
In this setting, the seminal work of Yamada and Watanabe \cite{Watanabe} established pathwise uniqueness of solutions under the condition that the diffusion coefficient $\sigma$ is $(1/2)$-H\"older continuous.
The study of pathwise uniqueness has been actively pursued for SDEs driven by symmetric $\alpha$-stable processes as well.
For the one-dimensional case, foundational results were established by Komatsu \cite{Komatsu} and Bass et al. \cite{Bass}.
Their work demonstrated that pathwise uniqueness holds under a H\"older continuity condition on the coefficient $\sigma$ of order $1/\alpha$, which is the natural analogue of the Brownian case.
This line of research was extended to the multi-dimensional setting by Tsuchiya \cite{Tsuchiya2}, who also considered non-Lipschitz coefficients.
Further generalizations have explored different aspects of the problem.
For instance, Fournier \cite{Fournier} established uniqueness conditions under various assumptions on the process's asymmetry and the coefficients' regularity and monotonicity.
The case where the stability index is small ($\alpha\in(0,1]$) was specifically addressed by Tsukada \cite{Tsukada}.
Beyond pathwise uniqueness, Kulik \cite{KULIK2} investigated the existence and properties of the unique weak solution, including the Gaussian bounds of its density function.
These studies on the existence and uniqueness of solutions naturally lead to the stability problem, which questions how the solution path changes when the coefficients of the SDE are perturbed. This problem is closely related to the analysis of numerical schemes, where, for instance, a series of studies on the convergence rate of Euler-Maruyama-type approximations have provided quantitative estimates under various coefficient regularities (\cite{Hashimoto}, \cite{Kaneko}, \cite{MiXu18}, \cite{PaTa17} and \cite{Yamada}, et al.).

Motivated by these quantitative approaches, this paper focuses directly on the stability problem.
A common application of stability results is to analyze the convergence of numerical schemes or approximation sequences.
For instance, one might consider a sequence of SDEs alongside, where for each $n\in\n$, $X^{(n)}$ is a solution to an approximating SDE:
\begin{align}
X_t^{(n)}=x_0^{(n)} + \int_0^t b_n\left(X_s^{(n)}\right) \rd s + \int_0^t \sigma_n \left(X_{s-}^{(n)}\right) \rd Z_s, \label{SDEn}
\end{align}
where $x_0^{(n)}$, $b_n:\real\to\real$ and $\sigma_n:\real\to\real$ are approximations of the original coefficients.

Quantifying the distance between the solutions $X$ of \eqref{SDE} and $X^{(n)}$ of \eqref{SDEn} is known as the stability problem.
While this paper's main theorems are formulated for two general processes, these results can be directly applied to analyze the convergence of such sequences.
%
This problem has a rich history.
For SDEs driven by semimartingales with Lipschitz coefficients, foundational stability results were provided by \'Emery \cite{Emery} and Protter \cite{Philip}.
The analysis was extended to the non-Lipschitz setting for Brownian-driven SDEs by Kawabata and Yamada \cite{Kawabata}.

Our work focuses on the stability of SDEs driven by $\alpha$-stable processes, a setting that introduces unique challenges due to the presence of jumps.
A significant portion of the existing literature has concentrated on the zero-drift case ($b=0$).
In this context, Hashimoto \cite{Hashimoto} proved the convergence in expectation of the time-supremum norm for a sequence of coefficients $(\sigma_n)_{n\in\n}$ converging uniformly to $\sigma$ under a Komatsu-type condition, though without providing convergence rates.
Subsequently, Hashimoto and Tsuchiya \cite{Tsuchiya} established explicit convergence rates in the $L^\infty$-norm.

Building upon these works, the present author in \cite{Nakagawa} extended the analysis by deriving convergence rates under a $L^p\left(\mu_{x_0,\sigma}^{\alpha,T}\right)$-norm. 
This norm, which will be defined precisely below, is designed to capture the behavior of the process more effectively than the standard $L^\infty$-norm.

While the zero-drift case is well-understood, the stability problem for non-zero drift is more complex.
A notable recent contribution is by Hao and Wu \cite{HaoWu}.
They established a quantitative stability estimate for SDEs with irregular drifts driven by a cylindrical $\alpha$-stable process, where the distance between solutions is measured in the total variation distance of the laws. Their framework accommodates a general class of multiplicative noise.
Our work differs from and complements the results in two significant aspects. First, whereas their analysis focuses on perturbations in the drift coefficient, our framework treats perturbations in both the drift and the jump coefficients simultaneously. Second, and more fundamentally, we establish stability in terms of the pathwise $L^{\alpha-1}(\Omega)$ distance between two solutions, which is a stronger notion of convergence than the convergence of laws. Our result thus provides a different type of quantitative estimate that captures the path-by-path proximity of solutions, which is often crucial for applications such as numerical scheme analysis.

The primary objective of this paper is to establish a comprehensive stability result that not only quantifies the impact of perturbations in the initial value $x_0$ and coefficients but also extends the analysis to a significantly broader and more practical class of perturbations.
Crucially, our framework extends previous results by considering perturbations to coefficients that are time-dependent, thus providing a more robust analysis of the solution's stability.
We extend the framework developed in \cite{Nakagawa} for the zero-drift case, aiming to derive convergence rates in a $L^p\left(\mu_{x_0,\sigma}^{\alpha,T}\right)$-norm that is well-suited to the process's behavior.

To this end, we measure the distance between coefficients using a specific weighted norm.
The motivation for this choice is to assign more weight to the region near the initial value $x_0$, where the process is most likely to be found, especially for small times.
We therefore introduce the finite measure $\mu_{x_0,\sigma}^{\alpha,t}$, defined by
\begin{align*}
\mu_{x_0,\sigma}^{\alpha,t}(\rd y)\coloneqq\frac{1}{t^{1/\alpha} \sigma^{1/\alpha}(y) }g^{(\alpha)} \left(\frac{y-x_0}{t^{1/\alpha} \sigma^{1/\alpha} (y)} \right)\rd y,
\end{align*}
where $g^{(\alpha)}$ is the density function of the stable random variable $Z_1$.
This function is globally comparable to a power-law function.
Specifically, there exist positive constants $c$ and $C$ such that for all $y\in\real$:
\begin{align*}
    c (|y|\vee 1)^{-1-\alpha} \le g^{(\alpha)}(y) \le C (|y|\vee 1)^{-1-\alpha}.
\end{align*}
We denote this relationship by $g^{(\alpha)}(y)\asymp (|y|\vee 1)^{-1-\alpha}$.
This measure $\mu_{x_0,\sigma}^{\alpha,t}$ is constructed to be directly comparable to the principal part of the transition density, $p_t^0(x_0, y)$ (see Subsection \ref{SDE_density_section}). Indeed, both share the same fundamental structure, incorporating the space-dependent coefficient $\sigma$. A key feature of this measure is that it decays as the distance $|y-x_0|$ increases. This follows from the decay property of $g^{(\alpha)}$ and the assumption that $\sigma$ is uniformly bounded from above and below, which ensures the argument of $g^{(\alpha)}$ grows linearly with $|y-x_0|$. This decay property ensures that $L^p\left(\mu_{x_0,\sigma}^{\alpha,t}\right)$-norm, defined below, effectively localizes the error analysis to the region near the initial point $x_0$.
\begin{align}
\|f\|_{L^p\left(\mu_{x_0,\sigma}^{\alpha,t}\right)}\coloneqq \left(\int_\real |f(y)|^p \mu_{x_0,\sigma}^{\alpha,t}(\rd y)\right)^{1/p}, \label{weighted_norm}
\end{align}
for any suitable function $f$ and $p,t>0$.
Our main theorem will bound the distance between two solutions in terms of the difference in initial values and this weighted norm for the coefficients.

While these previous results have provided crucial insights, their applicability remains significantly limited by the prevailing assumptions of a zero drift term and time-homogeneous coefficients. Overcoming these restrictions is not merely a matter of technical generalization but a necessary step to align the theory with both practical applications and deeper mathematical challenges. Many real-world phenomena modeled by SDEs—from mean-reverting financial assets to complex systems in physics and biology—exhibit a deterministic tendency or "drift" that fundamentally governs their qualitative behavior. A robust stability theory must therefore be able to incorporate non-zero drift to be practically meaningful. Similarly, the assumption of time-constant coefficients is an oversimplification for most non-autonomous systems, where environmental factors like volatility in finance or external forces in engineering naturally vary over time.
The inclusion of these features, however, presents substantial technical hurdles. The introduction of a drift term, in particular, creates additional integral terms within the classical Komatsu-type analysis that cannot be controlled by the estimates sufficient for the zero-drift case. Furthermore, handling time-dependent coefficients requires a conceptual shift in how the distance between functions is measured; a transient, large difference between coefficients should not necessarily imply instability, a subtlety that cannot be captured by simple supremum norms. This necessitates the development of a more flexible, integral-based measure of distance that is sensitive to the underlying dynamics of the process.

Our work directly confronts these challenges by developing a unified framework that successfully incorporates both non-zero drift and time-dependent coefficients.
This framework integrates the pioneering method of Komatsu \cite{Komatsu} for pathwise uniqueness with transition density estimates from Knopova and Kulik \cite{KULIK}, using a smooth approximation $\ude$ of the function $u(x)=|x|^{\alpha-1}$.
The key technical innovation of this paper is the extension of this method to handle SDEs with non-zero and time-dependent drift terms.
This extension is non-trivial.
While the original framework was effective for the zero-drift case (see \cite{Nakagawa}), the introduction of a drift term creates additional terms that cannot be controlled by existing estimates.
The genuine novelty of this paper lies in two main contributions that overcome this challenge.
First, and most critically, we establish a new, sharper derivative estimate for the mollified auxiliary function (inequality \eqref{ude'ineq}). This inequality is the technical core of our work. It provides a refined bound that is essential for decomposing the drift term into a manageable, recursive part and a part related to the direct coefficient difference. Without this estimate, the application of Gronwall's inequality to the recursive drift term would not be possible, and the entire stability argument would fail.
Second, building on this new estimate, we introduce a strategic optimization procedure for the auxiliary parameters $\eps>0$ and $\delta>1$.
These parameters are chosen not as constants, but as functions of the coefficient differences themselves.
This dynamic choice is essential for balancing the diverging terms in our stability inequality, ultimately leading to the explicit convergence rates in our main theorem.
This refined approach enables a unified treatment of both drift and jump coefficients, allowing us to derive quantitative stability bounds that hold beyond the zero-drift and time-independent setting.
Furthermore, we apply this stability analysis to investigate convergence in probability, leveraging properties of quasi-martingales \cite{Kurtz} to establish a rate of uniform convergence between solution paths in probability.
It is important to note the nature of our stability result.
Our framework is intentionally asymmetric: we begin with a well-understood, time-homogeneous baseline process and quantify its stability under time-dependent perturbations.
This approach models the common practical question of how a standard, established model behaves when subjected to more complex, non-autonomous influences.
Consequently, our main theorem should be interpreted as a sensitivity estimate for a baseline model against a broad class of time-dependent perturbations.

The remainder of this paper is organized as follows.
Section \ref{Notaion} establishes the notation and preliminary results used throughout the paper. We recall key estimates for the transition density of the solution to SDE \eqref{SDE}, based on the work of Knopova and Kulik \cite{KULIK}.
In Section \ref{distance}, we present and prove our first main result: a quantitative estimate for the $L^{\alpha-1}(\Omega)$ distance between two solutions. This section contains the core of our technical analysis, including the proof that relies on the new estimate for the approximating function.
Section \ref{distanceprob} is devoted to our second main result, which addresses the stability of solutions in the sense of convergence in probability. We introduce the concept of quasi-martingales and apply their properties to derive a rate of uniform convergence.
Finally, Section \ref{Appendices} contains the proofs of several technical lemmas that are essential for the arguments in the main sections, including the detailed derivation of our new estimate for the derivative of the mollified function.

\section{Preliminaries}\label{Notaion}
This section introduces the notation and foundational concepts used throughout the paper.

\subsection{General notation}
We begin by setting some general mathematical notation. For any real numbers $a, b$, we write $a\wedge b= \min\{a,b\}$ and $a\vee b= \max\{a,b\}$. The gamma function is denoted by $\Gamma(z)=\int_0^\infty x^{z-1}e^{-x}\rd x$ for $z>0$. For a real-valued function $f$, its uniform norm is given by $\displaystyle\|f\|_{\infty}=\sup_{x\in\real}|f(x)|$. The convolution of two functions $f$ and $g$ is written as $(f\ast g)(x) =\int_{\real} f(y)g(x-y)dy$.
\subsection{The symmetric $\alpha$-stable process}
The driving noise in our SDE is a symmetric $\alpha$-stable process $Z = \{Z_t\}_{t\in [0,T]}$ with a stability index $\alpha\in(1,2)$. We summarize its essential properties below.
The jump of the process at time $t>0$ is denoted by $\Delta Z_t \coloneqq Z_t-Z_{t-}$, with $\Delta Z_0\coloneqq 0$. The process $Z$ can be characterized by its jump measure. The associated Poisson random measure $N$ on $\mathcal{B}([0,T])\times \mathcal{B}(\real\setminus\{0\})$ is defined by counting the jumps:
$N(t,F) \coloneqq \sum_{0< s\le t} 1_F(\Delta Z_s)$ for $F\in \mathcal{B}(\real\setminus\{0\})$.
The intensity of these jumps is described by the L\'evy measure $\nu$, which is given by the expectation $\nu(F) \coloneqq \mathbb{E}[N(1,F)]$. Throughout this paper, the L\'evy measure of $Z$ is $\nu(\rd z) = c_\alpha |z|^{-1-\alpha} \rd z$. The compensated Poisson random measure, denoted by $\tilde{N}$, is then defined as the difference $\tilde{N}(\rd z,\rd s) \coloneqq N(\rd z,\rd s)-\nu(\rd z)\rd s$.

\subsection{The transition density function and its estimates for SDE \eqref{SDE}}\label{SDE_density_section}
Our stability analysis in Section \ref{distance} will require relating the expectation of coefficient differences to a weighted norm.
The key to this connection is the transition probability density $p_t(x_0,\cdot)$ of the baseline SDE \eqref{SDE}.
Therefore, we now introduce the essential results regarding the existence and properties of this density, established by Knopova and Kulik \cite{KULIK}.
For our setting where $\alpha\in(1,2)$, we adopt the assumptions of their Case A: the jump coefficient $\sigma$ is assumed to be strictly positive, bounded and H\"older continuous, while the drift coefficient $b$ is assumed to be continuous and bounded.
It shows that the density can be written as the sum of a principal part $p_t^0\left(x_0,y\right)$ and a residue term $r_t\left(x_0,y\right)$:
\begin{align}
    p_t\left(x_0,y\right)=p_t^0\left(x_0,y\right)+r_t\left(x_0,y\right). \label{SDE_density}
\end{align}
The principal part $p_t^0(x_0,\cdot)$ is the "frozen-coefficient" approximation, defined by
\begin{align*}
    p_t^0\left(x_0,y\right)\coloneqq \frac{1}{t^{1/\alpha} \sigma(y)^{1/\alpha}} g^{(\alpha)}\left( \frac{y-x_0}{t^{1/\alpha} \sigma(y)^{1/\alpha}} \right).
\end{align*}
Here, $g^{(\alpha)}$ is the density function of $Z_1$.
The theorem further provides a crucial estimate for the residue term $r_t$, controlled by
\begin{align*}
    \left|r_t\left(x_0,y\right)\right|\le C p_t^0\left(x_0,y\right)V_t\left(y-x_0\right),\quad t\in(0,T],\ y\in\real,
\end{align*}
where $C$ is a constant and for any $\kappa\in(0,\eta)$, $V_t$ is a weight function defined by
\begin{align*}
    V_t(z) = 
    \begin{cases} 
    t^{\kappa/\alpha} + t^{\kappa}, & \textrm{if } |z| \le t^{1/\alpha}, \\
    |z|^{\kappa} + t^{\kappa}, & \textrm{if } t^{1/\alpha} < |z| \le 1, \\
    1 + t^{\kappa}, & \textrm{if } |z| > 1.
    \end{cases}
\end{align*} 
As a direct consequence of this decomposition $p_t = p_t^0 + r_t$ and the estimate on the residue term, it follows that the density $p_t$ is of the same order as its principal part $p_t^0$.
Indeed, this is stated more directly as a two-sided estimate in Theorem 2.5 in \cite{KULIK}, which shows that for some constants $m, M > 0$,
\begin{align}
    m p_t^0\left(x_0,y\right)\le p_t\left(x_0,y\right) \le M p_t^0\left(x_0,y\right) \label{SDE_density_estimate}
\end{align}
holds for $t\in(0,T]$.
The upper bound provided by this estimate will be a convenient tool in our proofs of Theorem \ref{mainthm}, \ref{mainthm2}.
Given the asymptotic behavior $g^{(\alpha)}(y) \asymp (|y|\vee 1)^{-1-\alpha}$, we can replace $g^{(\alpha)}$ in the preceding bounds with the function $G^{(\alpha)}(y) = (|y|\vee 1)^{-1-\alpha}$ to obtain more explicit estimates (see Remark 2.4 in \cite{KULIK}).

\section{Distance between solutions in the $L^{\alpha-1}(\Omega)$-norm}\label{distance}
In this section, we present our first main result: a quantitative stability estimate for the solutions of SDE \eqref{SDE}.
Our result generalizes previous findings in two significant directions.
First, whereas Hashimoto and Tsuchiya \cite{Tsuchiya} measured the distance between coefficients using the supremum norm, our estimate is formulated in terms of the more refined $L^p\left(\mu_{x_0,\sigma}^{\alpha,T}\right)$-norm introduced in the previous section.
Second, our framework incorporates non-zero drift terms, extending the zero-drift analysis from \cite{Nakagawa}.
The proof strategy involves a delicate balancing act.
We will employ Komatsu's method, which uses an auxiliary function $\ude$ controlled by parameters $\eps$ and $\delta$.
Applying It\^o's formula yields a fundamental inequality, but this inequality contains coefficients that diverge as $\eps \to +0$ or $\delta \to \infty$.
The core idea of our proof is to strategically choose these parameters as functions of the norms measuring the differences between the coefficients.
This optimization process balances the diverging terms against the smallness of the errors, and the explicit convergence rates presented in the following theorem are a direct consequence of this procedure.
In addition, it is important to emphasize that the restriction $\alpha > 1$ is not merely a consequence of the cited transition density theory introduced in Subsection \ref{SDE_density_section}, but is intrinsic to our Komatsu-type approach. Our proof is built upon an auxiliary function, $\ude$, which is a mollified version of the function $u(x)=|x|^{\alpha-1}$. The entire framework relies on the exponent $\alpha-1$ being positive, which ensures that this underlying function is regular at the origin and suitable for measuring pathwise distance. The case $\alpha \le 1$, known as the supercritical case, would require a fundamentally different methodology and is thus outside the scope of this paper.

We are now in a position to state our first main result.
The following theorem provides an explicit bound on the $L^{\alpha-1}(\Omega)$ distance between the solution $X$ of our baseline SDE \eqref{SDE} and the solution $\tilde{X}$ of a more general SDE with time-dependent coefficients.
As a consistency check, we note that in the special case where $b=\tilde{b}=0$, our result recovers the convergence rate established in \cite[Theorem 31]{Nakagawa}.

\begin{theorem}\label{mainthm}
Let $Z$ be a one-dimensional symmetric $\alpha$-stable process with $\alpha \in (1,2)$. Consider two processes, $X$ and $\tilde{X}$, satisfying the SDEs
\begin{align*}
X_t = x_0 + \int_0^t b(X_s) \rd s + \int_0^t \sigma(X_{s-}) \rd Z_s \quad \text{and} \quad \tilde{X}_t = \tilde{x}_0 + \int_0^t \tilde{b}(s,\tilde{X}_s) \rd s +  \int_0^t \tilde{\sigma}(s,\tilde{X}_{s-}) \rd Z_s,
\end{align*}
for $t\in[0,T]$ with initial values $x_0, \tilde{x}_0\in\real$.

Suppose the coefficients satisfy the following conditions for some constants $K, k, K'>0$ and $\tilde{\eta} \in [1/\alpha,1]$:
\begin{itemize}
    \item[(i)] Drift coefficients:
    The coefficient $b$ of the baseline process $X$ is continuous and bounded: $|b(x)|\le K$ for all $x\in\real$.
    The coefficient $\tilde{b}$ of the perturbed process $\tilde{X}$ is spatially Lipschitz continuous with a time-dependent Lipschitz constant $f_{\tilde{b}}$ that is integrable over $[0,T]$:
    \begin{align*}
       \left|\tilde{b}(t,x)-\tilde{b}(t,y)\right|\le f_{\tilde{b}}(t) |x-y|, \quad \int_0^T f_{\tilde{b}}(s) \rd s<\infty. 
    \end{align*}
    \item[(ii)] Jump coefficients:
    The coefficient $\sigma$ of $X$ is strictly positive, bounded, and H\"older continuous:
    \begin{align*}
        k\le \sigma(x)\le K,\quad |\sigma^\alpha(x)-\sigma^\alpha(y)|\le K|x-y|^\eta.
    \end{align*}
    The coefficient $\tilde{\sigma}$ of the perturbed process $\tilde{X}$ satisfies the following conditions.
    It is bounded in magnitude by a function $g_{\tilde{\sigma}}$ that is square-integrable on $[0, T]$:
    \begin{align*}
        |\tilde{\sigma}(t,x)|\le g_{\tilde{\sigma}}(t), \quad
        \left(\int_0^T g_{\tilde{\sigma}}^2(s) \rd s\right)^{1/2}<\infty.
    \end{align*}
    Furthermore, it is spatially $\tilde{\eta}$-H\"older continuous with a time-dependent Lipschitz constant $f_{\tilde{\sigma}}$ that belongs to $L^\alpha[0,T]$: 
    \begin{align*}
       |\tilde{\sigma}(t,x)-\tilde{\sigma}(t,y)|\le f_{\tilde{\sigma}}(t)|x-y|^{\tilde{\eta}},\quad
       \left(\int_0^T f_{\tilde{\sigma}}^\alpha(s) \rd s\right)^{1/\alpha}<\infty
    \end{align*}
    \item[(iii)] The distances between the coefficients are small. 
    These distances are measured by the quantities $B$ and $S$, defined by integrals weighted by the transition probability density $p_t(x_0,\cdot)$ of the baseline process $X$ (as described in \eqref{SDE_density}):
    \begin{align*}
       B \coloneqq \int_0^T \int_\real \left|b(y)-\tilde{b}(s,y)\right| p_s\left(x_0,y\right) \rd s \rd y < 1, \quad S \coloneqq \left(\int_0^T \int_\real \left|\sigma(y)-\tilde{\sigma}(s,y)\right|^\alpha p_s\left(x_0,y\right) \rd s \rd y \right)^{1/\alpha} < 1. 
    \end{align*}
\end{itemize}
Then, there exists a positive constant $C = C\left(T, \alpha, K, k, \int_0^T f_{\tilde{b}}(s) \rd s, \int_0^T f_{\tilde{\sigma}}^\alpha(s) \rd s\right)$ such that
\begin{align*}
\sup_{t\in[0,T]}\mathbb{E}\left[|X_t-\tilde{X}_t|^{\alpha-1}\right] \le 
\begin{cases}
    C\left(|x_0-\tilde{x}_0|^{\alpha-1}+ \max\left\{B^{(\alpha\tilde{\eta}-1)/(\alpha\tilde{\eta}-\alpha+1)}, S^{\alpha-1/\tilde{\eta}}\right\}\right\} & \text{if } \tilde{\eta} \in(1/\alpha,1], \\
	C\left(|x_0-\tilde{x}_0|^{\alpha-1}+ \left(\log\frac{1}{\max\{B, S\}}\right)^{-1}\right) & \text{if } \tilde{\eta} = 1/\alpha.
\end{cases}
\end{align*}
\end{theorem}
It is important to interpret the stability estimate of Theorem \ref{mainthm} as a measure of overall stability across the entire time horizon. The resulting estimate does not describe how the error evolves from moment to moment, but rather captures the total accumulated effect of the coefficient perturbations over the whole interval $[0, T]$.
We add several remarks concerning the statement and assumptions of Theorem \ref{mainthm}.
\begin{remark}
\begin{enumerate}
    \item \textbf{(On the interpretation of the coefficient distances $B$ and $S$)}
    It is important to clarify the interpretation of the quantities $B$ and $S$. They do not measure the uniform closeness of the coefficients over all space and time (such as in a supremum norm). Rather, they define a notion of distance that is weighted by the law of the baseline process $X$.
    This feature of the framework is intentional and particularly effective for handling time-dependent perturbations. It implies that the pointwise differences between the coefficients $(b,\sigma)$ and $(\tilde{b},\tilde{\sigma})$ are permitted to be large, so long as these differences occur in regions where the baseline process $X$ is unlikely to be found (i.e., where its transition density $p_t(x_0, \cdot)$ is small). In essence, $B$ and $S$ measure the discrepancy between the coefficients primarily in the regions most relevant to the dynamics of the baseline process.
    Consequently, the theorem should be understood not as a constructive approximation result, but as a robustness statement: it asserts that if the discrepancy between the two sets of coefficients is small in this weighted sense, then the corresponding solution paths will also be close. This notion of distance makes our stability result applicable to a wider class of perturbations than would be possible with a uniform measure of distance. In the specific time-homogeneous case considered in Corollary \ref{Cor_time-homogeneous}, these general quantities $B$ and $S$ naturally reduce to more standard weighted $L^p$-norms.
    \item \textbf{(On the coefficients of the baseline SDE \eqref{SDE})} The assumptions on the coefficient $\sigma$ (i.e., being strictly positive, bounded, and H\"older continuous) are required to guarantee the existence and two-sided estimates of its transition density function, as established in \cite{KULIK}. Note that the H\"older continuity of $\sigma^\alpha$ with exponent $\eta$ implies that $\sigma$ itself is H\"older continuous with exponent $\eta/\alpha$.
    \item \textbf{(On the coefficients of the perturbed SDE)} 
    For the drift coefficient $\tilde{b}$, we require spatial Lipschitz continuity.
    While this is a stronger regularity condition compared to the H\"older continuity assumed for the $b$, we do not require $\tilde{b}$ to be bounded. 
    The upper bound on $|\tilde{\sigma}|$ by a square-integrable function $g_{\tilde{\sigma}}$ is a technical assumption needed to ensure that a certain stochastic integral term in the proof is a true martingale.
    The constant $C$ in our final estimate is independent of this upper bound.
    In addition to these conditions, if $\tilde{b}$ and $\tilde{\sigma}$ are time-independent, pathwise uniqueness of the solution $\tilde{X}$ is guaranteed (see, e.g., \cite{Fournier}).
    \item \textbf{(Asymmetric structure around a baseline model)}
    The stability estimate is asymmetric. The roles of the baseline process $X$ and the perturbed process $\tilde{X}$ are not interchangeable. The distances between coefficients, $B$ and $S$, are always measured using the transition probability density $p_t(x_0,\cdot)$ of the baseline process $X$. This structure is intentional and reflects the core purpose of the theorem: to provide a sensitivity analysis of a well-understood, time-homogeneous model ($X$) against a wide range of time-dependent perturbations ($\tilde{X}$). This framework is particularly relevant in applications where one wants to assess the robustness of a standard model against realistic, non-autonomous environmental changes.
    \item \textbf{(On the solution framework)}
    Our main theorem provides a stability estimate between two solutions, $X$ and $\tilde{X}$. We clarify here the precise solution concepts and the underlying assumptions on their existence.
    \begin{itemize}
        \item For the baseline process $X$: Under our assumptions, the existence of a weak solution and its transition probability density is guaranteed by the work of Knopova and Kulik \cite{KULIK}. We do not assume pathwise uniqueness for $X$, as it is known to fail if the H\"older exponent $\eta<1/\alpha$ (see Bass et al. \cite{Bass}).
        \item For the perturbed process $\tilde{X}$: Our theorem is formulated under the hypothesis that a strong solution $\tilde{X}$ exists. The question of strong well-posedness for SDEs driven by stable processes with H\"older continuous coefficients is a highly non-trivial problem at the forefront of current research.
        To place this in context, recent significant results have been established under different sets of assumptions. For instance, Chen, et al. \cite{ChZhZh} proved strong well-posedness when the diffusion coefficient $\tilde{\sigma}$ is Lipschitz continuous and the drift $b$ is H\"older continuous (with exponent $\beta>1-\alpha/2$). In the Brownian motion case $(\alpha=2)$, Hutzenthaler and Jentzen \cite{HuJe15} have shown the existence of strong solutions even when the diffusion coefficient is merely H\"older continuous. However, extending such results to the $\alpha$-stable case with a H\"older continuous diffusion coefficient remains a challenging open problem.
    \end{itemize}
    Therefore, our main result should be interpreted as follows: if a strong solution $\tilde{X}$ to the perturbed SDE exists on a given probability space, then for any weak solution $X$ of the baseline SDE on the same space, their distance is bounded by the explicit rate given in Theorem \ref{mainthm}. Proving the existence of $\tilde{X}$ under our specific, time-dependent H\"older assumptions is a substantial research problem in its own right and is considered beyond the scope of this paper.
\end{enumerate}
\end{remark}

\subsection{Komatsu's method and properties of the mollifying Function}
Our proof of Theorem \ref{mainthm} is based on a variation of a method originally developed by Komatsu \cite{Komatsu} for establishing pathwise uniqueness.
The core idea is to evaluate the quantity $\left|X_t - \tilde{X}_t\right|^{\alpha-1}$ not directly, but through a smooth approximation.

While pathwise uniqueness for multi-dimensional $\alpha$-stable driven SDEs has been studied (e.g., Tsuchiya \cite{Tsuchiya2}), extending our quantitative stability analysis to higher dimensions presents a significant challenge.
Tsuchiya's work also utilizes an approximation of the function $u(\mathbf{x})=|\mathbf{x}|^{\alpha-1}$, but the approximation is constructed via Bessel and hypergeometric functions, which is very different from our direct mollification approach.
A key technical innovation of our paper is the new, sharp derivative estimate for the mollified function (inequality \eqref{ude'ineq}). Although similar estimates are obtained using the approximation based on Bessel and hypergeometric functions, it fails to satisfy the localization property required in Lemma \ref{Lude}, which is essential for deriving the offset-free estimate of the difference in coefficients.
This limitation stems from the fact that the kernel of Riesz potential of $L_\alpha$ corresponds to $|\cdot|^{\alpha-d}$ in the $d$-dimensional case.
Consequently, applying $L_\alpha$ to $|\mathbf{x}|^{\alpha-1}$ inevitably leaves a tail of order $|\mathbf{x}|^{-1}$ when $d \ge 2$.
In the evaluation of the coefficient difference, this tail interacts with the $\tilde{\eta}$-H\"older continuity of the coefficient $\tilde{\sigma}$, leading to an error term of the form $C \int_0^t f_{\tilde{\sigma}}(s) \left|X_s-\tilde{X}_s\right|^{\alpha\tilde{\eta}-1} \rd s$ in the stability estimate.
Unless both the coefficients $\tilde{b}$ and $\tilde{\sigma}$ are Lipschitz continuous ($\tilde{\eta}=1$), this order is strictly lower than our target distance $\left|X_t-\tilde{X}_t\right|^{\alpha-1}$, which precludes the use of Gronwall's inequality and makes it difficult to obtain the desired quantitative bound.

Thus, our analysis in this paper is restricted to the one-dimensional case. We follow the approach in \cite{Tsuchiya,Nakagawa} and define the approximation by mollifying the function $u(x)=|x|^{\alpha-1}$.
Specifically, let $\pde$ be a smooth mollifier and define $u_{\delta,\eps} \coloneqq u \ast \pde$.
The key properties of this function and its derivatives, which are crucial for applying Ito's formula and controlling the resulting terms, are summarized in the following lemma.

\begin{Lem}\label{abstruct}
Let $u(x)=|x|^{\alpha-1}$ and $\alpha\in(1,2)$. For any $\eps >0$ and $\delta>1$, one can construct a smooth function $\pde$ such that:
\begin{itemize}
    \item[(i)] $\pde$ is supported on $\left[\eps\delta^{-1}, \eps\right]$ and satisfies $\int_{\eps\delta^{-1}}^{\eps} \pde(y) dy=1$, and is bounded by
    \begin{align*}
0\le \pde(x) \le 2(x \log \delta)^{-1}\quad \textrm{for~}x\in\left[\eps\delta^{-1}, \eps\right] .
\end{align*}
    \item[(ii)] The mollified function $u_{\delta,\eps} = u \ast \pde$ is of class $C^2$ and satisfies the following inequalities for all $x\in\real$:
    \begin{align}
        |x|^{\alpha-1} &\le \eps^{\alpha-1} + u_{\delta,\eps}(x), \label{udeineq1} \\
        u_{\delta,\eps}(x) &\le |x|^{\alpha-1} + \eps^{\alpha-1}, \label{udeineq2} \\
        |u'_{\delta,\eps}(x)| &\le 2^{2-\alpha}(\alpha-1)|x|^{\alpha-2}\1_{[-2\eps,2\eps]^c}(x) + \frac{2^{3-\alpha}\delta(1-\delta^{-1})^{\alpha-1}}{\eps^{2-\alpha}\log \delta}\1_{[-2\eps,2\eps]}(x). \label{ude'ineq}
    \end{align}
\end{itemize}
\end{Lem}

Inequalities \eqref{udeineq1} and \eqref{udeineq2} are standard mollification estimates, recalled from \cite{Nakagawa}.
In sharp contrast, the estimate \eqref{ude'ineq} for the derivative is a new technical result that is central to this paper.
Its proof, which is crucial for handling the drift term introduced in our setting, is provided in Section \ref{M-martingale}.

Finally, we will also use the following identity, which describes the action of the generator $L_\alpha$ on $u_{\delta,\eps}$. A proof can be found in \cite[Lemmas 34 and 35]{Nakagawa}.

\begin{Lem}\label{Lude}
The function $u_{\delta,\eps}$ satisfies
\begin{align*}
L_{\alpha} u_{\delta,\eps}(\theta) = C_\alpha \pde(\theta) \quad \text{for each } \theta \neq 0,
\end{align*}
where the constant $C_\alpha$ is given by $C_\alpha=-2 \alpha \sin\left(\frac{\alpha\pi}{2}\right)\cot \left(\frac{\alpha \pi}{2}\right)\Gamma(\alpha+1)>0$.
\end{Lem}

We note that $L_\alpha u_{\delta,\eps}$ is well-defined even though $u_{\delta,\eps}$ is not a rapidly decreasing function. A justification for this is provided in Section \ref{Lude_well-defined}.

\subsection{Proof of Theorem \ref{mainthm}}
The proof strategy follows the framework established in \cite{Nakagawa}, which is a variation of Komatsu's method. A crucial ingredient in our proof is the availability of two-sided estimates for the transition density $p_t(x_0, \cdot)$ of the solution $X_t$ of SDE \eqref{SDE}, as outlined in Section \ref{SDE_density_section}.
\begin{proof}
The proof proceeds in several steps.
First, we apply Ito's formula to the smooth mollified function $\ude(Y_t)$, where $Y_t\coloneqq X_t-\tilde{X}_t$ is the difference between the two solutions.
This will decompose the process into a martingale part and several integral terms.
Second, we will estimate each integral term using the properties of $\ude$ from Lemma \ref{abstruct} and the transition density estimates.
Finally, we will optimize the choice of the parameters $\delta$ and $\eps$ to obtain the desired convergence rate.
Applying inequality \eqref{udeineq1} from Lemma \ref{abstruct} to $Y_t$, we have
\begin{align}
|Y_t|^{\alpha-1}\le \eps^{\alpha-1}+\ude\left(Y_t\right).\label{abstructYt}
\end{align}
Next, we apply It\^o's formula for jump processes (see, e.g., \cite[Theorem 4.4.7]{Applebaum} to the process $\ude(Y_t)$. Combined with the L\'evy-It\^o decomposition (see, e.g., \cite[Theorem 2.4.16]{Applebaum}) and inequality \eqref{udeineq2}, this yields the following decomposition:
\begin{align}
    \ude\left(Y_s\right)
    &=\ude\left(Y_0\right)+M_t^{\delta,\eps}+I_t^{\delta,\eps,b}+I_t^{\delta,\eps,\sigma}\notag\\
    &\le \left|x_0-\tilde{x}_0\right|^{\alpha-1}+\eps^{\alpha-1}+M_t^{\delta,\eps}+I_t^{\delta,\eps,b}+I_t^{\delta,\eps,\sigma},\label{udeYt}
\end{align}
where the terms are defined as follows:
\begin{align*}
    M_t^{\delta,\eps}&=\int_0^t\int_{\real\setminus\{0\}}\left\{\ude\left(Y_{s-}+\left(\sigxs\right)z\right)-\ude \left(Y_{s-}\right)\right\}\tilde{N}(\rd s,\rd z),\\
    I_t^{\delta,\eps,b}&=\int_0^t \ude'\left(Y_s\right)\left(b\left(X_s\right)-\tilde{b}\left(s,\tilde{X}_s\right)\right)\rd s,\\
    I_t^{\delta,\eps,\sigma}&=\int_0^t\int_{\real\setminus\{0\}}\left\{\ude\left(Y_{s-}+\left(\sigxs\right)z\right)-\ude \left(Y_{s-}\right)+\1_{\{|z|\le 1\}}\ude'\left(Y_s\right)\left(\sigxs\right)z\right\}\frac{c_\alpha}{|z|^{1+\alpha}}\rd s \rd z.
\end{align*}
The process $M^{\delta,\eps}\coloneqq \left(M_t^{\delta,\eps}\right)_{0\le t \le T}$ is a true martingale. This can be verified by following the argument in \cite[Subsection 5.1]{Nakagawa}, which involves decomposing the process into its $L^1$ and $L^2$ martingale parts.
This argument relies on the Lipschitz continuity of $\ude$ (which follows from the boundedness of its derivative $u'_{\delta,\varepsilon}$) and the boundedness of the jump coefficients (see Subsection \ref{M-martingale}).
While the boundedness of $u'_{\delta,\varepsilon}$ was already sufficient for this purpose in the zero-drift case, our new, sharper estimate \eqref{ude'ineq} for $|u'_{\delta,\varepsilon}(x)|$ plays a crucial role in the next step of the proof: controlling the drift term $I_t^{\delta,\varepsilon,b}$.
It is this refined estimate that makes our extension to the non-zero drift case possible.

We now turn to the estimation of the drift term $I_t^{\delta,\eps,b}$, which represents the main challenge in extending the stability analysis to the non-zero drift case.
Our strategy is to handle the two sources of difference—the perturbation of the coefficient itself and the deviation of the solution paths—separately.
To this end, we apply the triangle inequality to split the integral into two parts:
\begin{align}
    \left|I_t^{\delta,\eps,b}\right|
    \le \int_0^t \left|\ude'\left(Y_s\right)\right|\left|b\left(X_s\right)-\tilde{b}\left(s,X_s\right)\right|\rd s 
        + \int_0^t \left|\ude'\left(Y_s\right)\right|\left|\tilde{b}\left(s,X_s\right)-\tilde{b}\left(s,\tilde{X}_s\right)\right|\rd s \label{Itb}
\end{align}
The first integral, corresponding to the direct coefficient difference, is bounded using our key estimate for $|\ude'|$ from inequality \eqref{ude'ineq}.
This yields:
\begin{align}
    &\int_0^t \left|\ude'\left(Y_s\right)\right|\left|b\left(X_s\right)-\tilde{b}\left(s,X_s\right)\right|\rd s \notag\\
    &\le 
       2^{2-\alpha}(\alpha-1) \int_0^t \1_{[-2\eps,2\eps]^c}(Y_s)\left|Y_s\right|^{\alpha-2}\left|b\left(X_s\right)-\tilde{b}\left(s,X_s\right)\right|\rd s \notag\\
        &\qquad+
        \frac{2^{3-\alpha}\delta\left(1-\delta^{-1}\right)^{\alpha-1}}{\eps^{2-\alpha}\log \delta}\int_0^t \1_{[-2\eps,2\eps]}(Y_s)\left|b\left(X_s\right)-\tilde{b}\left(s,X_s\right)\right|\rd s ,\notag\\
    &= \left\{\frac{\alpha-1}{\eps^{2-\alpha}}
        +
        \frac{2^{3-\alpha}\delta\left(1-\delta^{-1}\right)^{\alpha-1}}{\eps^{2-\alpha}\log \delta} \right\}\int_0^t\left|b\left(X_s\right)-\tilde{b}\left(s,X_s\right)\right|\rd s. \label{Itb_term1}
\end{align}
This term will be controlled later by taking the expectation, which will connect it to the coefficient distance $B$.
The second integral, arising from the path deviation $Y_s$, is handled by leveraging the Lipschitz continuity of $\tilde{b}$ along with the same derivative estimate.
This will produce a term involving $|Y_s|$ which will be controlled using Gronwall's inequality at the final stage of the proof.
Without the sharp control provided by inequality \eqref{ude'ineq}, we would find it difficult to obtain an integrand in this form, which is necessary for the subsequent application of Gronwall's inequality.
This leads to the bound:
\begin{align}
    &\int_0^t \left|\ude'\left(Y_s\right)\right|\left|\tilde{b}\left(s,X_s\right)-\tilde{b}\left(s,\tilde{X}_s\right)\right|\rd s\notag\\
    &\le \int_0^t \left|\ude'\left(Y_s\right)\right| f_{\tilde{b}}(s) \left|Y_s\right| \rd s,\notag\\
    &\le \int_0^t\left\{2^{2-\alpha}(\alpha-1) f_{\tilde{b}}(s)\left|Y_s\right|^{\alpha-1}\1_{[-2\eps,2\eps]^c}(Y_s)+\frac{2^{3-\alpha}\delta(1-\delta^{-1})^{\alpha-1}}{\eps^{2-\alpha}\log \delta}f_{\tilde{b}}(s)\left|Y_s\right|\1_{[-2\eps,2\eps]}(Y_s)\right\}\rd s,\notag\\
    &\le 2^{2-\alpha}(\alpha-1)C_\alpha \int_0^t f_{\tilde{b}}(s)\left|Y_s\right|^{\alpha-1} \rd s + \frac{2^{4-\alpha} \eps^{\alpha-1} \delta(1-\delta^{-1})^{\alpha-1}}{\log \delta}\int_0^T f_{\tilde{b}}(s)\rd s, \label{Itb_term2}
\end{align}

Having addressed the drift term, we proceed to analyze the jump compensator term, $I_t^{\delta,\eps,\sigma}$.
The argument to bound this term is identical to that used in the proof of \cite[Theorem 31]{Nakagawa}.
This is because the structure of the jump term analysis, which relies on a change of variables and the properties of the mollifier $\pde$, is not affected by the presence of the drift term.
We outline the main steps.
First, by applying a change of variables and Lemma \ref{Lude}, the term is rewritten using the mollifier $\pde$:
\begin{align*}
    I_t^{\delta,\eps,\sigma} = C_{\alpha}\int_0^t \frac{\left|\sigxs\right|^{1+\alpha}}{\sigxs}\pde(Y_s)\rd s.
\end{align*}
By moving the absolute value inside the integral, which is a standard property of integration, we obtain the bound:
\begin{align*}
I_t^{\delta,\eps,\sigma}\le C_{\alpha} \int_0^t \left|\sigxs\right|^\alpha \pde(Y_s)\rd s. 
\end{align*}
To estimate this latter expression, we apply the triangle inequality to the term $\left|\sigma\left(X_{s}\right)-\tilde{\sigma}\left(s,\tilde{X}_{s}\right)\right|^\alpha$ and utilize the explicit upper bound on $\pde$ from Lemma \ref{abstruct}(i), along with the $\tilde{\eta}$-H\"older continuity of the coefficient $\tilde{\sigma}$. This estimation procedure is precisely the same as in the zero-drift case analyzed in \cite{Nakagawa}. Following the computations therein, we arrive at the final inequality:
\begin{align}
I_t^{\delta,\eps,\sigma}
                     &\le
                     \frac{2^\alpha C_\alpha \eps^{\alpha\tilde{\eta}-1}}{\log \delta}\int_0^t f_{\tilde{\sigma}}^\alpha(s)\rd s 
                     + \frac{2^\alpha C_\alpha}{\log \delta}\left(\frac{\delta}{\eps}\right)^{ }\int_0^t \sigab \rd s,\label{Itdeltaeps}    
\end{align}
For the sake of brevity, we have omitted the detailed steps of this calculation and refer the reader to \cite{Nakagawa} for the complete argument.

With all components of the Ito decomposition \eqref{udeYt} now bounded, we are now ready to combine the estimates for the components of \eqref{udeYt}.
First, note that since $\alpha\in(1,2)$ and $\tilde{\eta} \in [1/\alpha,1]$, the exponent $\alpha\tilde{\eta}-1$ is always non-negative.
Substituting the bounds for $\left|I_t^{\delta,\eps,b}\right|$ (from \eqref{Itb}, \eqref{Itb_term1}, \eqref{Itb_term2}) and $\left|I_t^{\delta,\eps,\sigma}\right|$ (from \eqref{Itdeltaeps}) into the main decomposition \eqref{udeYt}, and then using the initial estimate \eqref{abstructYt}, we arrive at an integral inequality for $\left|Y_t\right|^{\alpha-1}$.
This inequality is recursive in nature, as the term $\int_0^t f_{\tilde{b}}(s)|Y_s|^{\alpha-1}\rd s$ (from \eqref{Itb_term2}) appears on the right-hand side.
We handle this recursive structure by applying Gronwall's inequality, which yields an explicit bound for $|Y_t|^{\alpha-1}$ in terms of the initial difference, the martingale term, and the integrals of coefficient differences.
The next step is to take the expectation.
Since $M_t^{\delta,\eps}$ is a martingale, its expectation vanishes.
Applying Fubini's theorem to the remaining terms and taking the supremum over $t \in [0,T]$, we obtain the following explicit inequality:
\begin{align}
\sup_{0\le t\le T}\e\left[\left|X_t-\tilde{X}_t\right|^{\alpha-1}\right]&\le \exp\left\{2^{2-\alpha}(\alpha-1)C_\alpha \int_0^T f_{\tilde{b}}(s)\rd s\right\}
            \bigg[\left|x_0-\tilde{x}_0\right|^{\alpha-1}+\eps^{\alpha-1}\notag\\
        &\quad +\left\{\frac{\alpha-1}{\eps^{2-\alpha}}
        +
        \frac{2^{3-\alpha}\delta\left(1-\delta^{-1}\right)^{\alpha-1}}{\eps^{2-\alpha}\log \delta} \right\}\int_0^T\e\left[\left|b\left(X_s\right)-\tilde{b}\left(s,X_s\right)\right|\right]\rd s\notag\\
        &\quad + \frac{2^{4-\alpha} \eps^{\alpha-1} \delta(1-\delta^{-1})^{\alpha-1}}{\log \delta}\int_0^T f_{\tilde{b}}(s) \rd s\notag\\ 
        &\quad+\frac{2^\alpha C_\alpha \eps^{\alpha\tilde{\eta}-1}}{\log \delta}\int_0^T f_{\tilde{\sigma}}^\alpha(s)\rd s + \frac{2^\alpha C_\alpha}{\log \delta}\left(\frac{\delta}{\eps}\right)\int_0^T \e\left[\sigab\right] \rd s\bigg].\label{alpha-1_est}
\end{align}
The inequality \eqref{alpha-1_est} provides a powerful general bound.
The final crucial step is to relate the time-averaged expectations on the right-hand side to the coefficient distance measures $B$ and $S$ defined in assumption (iii) of Theorem \ref{mainthm}.
This connection is forged by leveraging the fact that the process $X_t$ admits a transition density $p_t(x_0, \cdot)$, which allows us to apply Fubini-Tonelli's theorem.
For the drift coefficient term in \eqref{alpha-1_est}, we have:
\begin{align*}
    \int_0^T \e\left[\left|b\left(X_s\right)-\tilde{b}\left(s, X_s\right)\right|\right] \rd s = \int_0^T \int_\real \left|b(y)-\tilde{b}(s, y)\right| p_s(x_0, y) \rd s \rd y.
\end{align*}
The right-hand side is precisely the definition of $B$.
Similarly, for the jump coefficient term, we have:
\begin{align*}
    \int_0^T \e\left[\left|\sigma\left(X_s\right)-\tilde{\sigma}\left(s, X_s\right)\right|^\alpha\right] \rd s = \int_0^T \int_\real \left|\sigma(y)-\tilde{\sigma}(s, y)\right|^\alpha p_s(x_0, y) \rd s \rd y.
\end{align*}
The right-hand side is equal to $S^\alpha$.
Then, inequality \eqref{alpha-1_est} implies:
\begin{align}
\sup_{0\le t \le T}\e\left[\left|X_t-\tilde{X}_t\right|^{\alpha-1}\right]
        &\le  \hat{C}_{\alpha,T}
            \bigg[\left|x_0-\tilde{x}_0\right|^{\alpha-1}+\eps^{\alpha-1}\notag\\
        &\quad +\left\{\frac{1}{\eps^{2-\alpha}}
        +
        \frac{\delta\left(1-\delta^{-1}\right)^{\alpha-1}}{\eps^{2-\alpha}\log \delta} \right\}B\notag\\
        &\quad + \frac{\eps^{\alpha-1} \delta(1-\delta^{-1})^{\alpha-1}}{\log \delta}\notag\\ 
        &\quad+\frac{\eps^{\alpha\tilde{\eta}-1}}{\log \delta}+\frac{\delta}{\eps \log \delta}S^\alpha,\label{alpha-1_norm_est}
\end{align}
where $\hat{C}_{\alpha,T}=\exp\left\{2^{2-\alpha}(\alpha-1)C_\alpha \int_0^T f_{\tilde{b}}(s)\rd s\right\}\max\left\{2^{3-\alpha},2^{4-\alpha}\int_0^T f_{\tilde{b}}(s)\rd s,2^\alpha C_\alpha \int_0^T f_{\tilde{\sigma}}^\alpha(s)\rd s, 2^\alpha C_\alpha \right\}$.

The inequality \eqref{alpha-1_norm_est} is the final estimate before the optimization of the free parameters $\eps>0$ and $\delta>1$.
The core challenge, and the central idea of the proof, lies in the structure of this bound.
Observe that several coefficients on the right-hand side, such as those proportional to $\eps^{-(2-\alpha)}$ or $(\log \delta)^{-1}$, diverge as $\eps\to+0$ or as $\delta\to+\infty$ approaches its limits.
Conversely, other terms, like $\eps^{\alpha-1}$, vanish.
The key to resolving this tension is to choose $\eps$ and $\delta$ not as fixed constants, but as functions of the small coefficient differences, $B$ and $S$. This strategic choice creates a trade-off: we balance the diverging coefficients against the smallness of the norms $B$ and $S$. For instance, a term of the form $B/\varepsilon^{2-\alpha}$ can be made to converge to zero if we let $\eps$ be an appropriate power of $B$ (e.g., $\varepsilon = B^p$ for a suitable $p$).
The optimal way to perform this balancing depends on the exponents in the inequality.
Crucially, the term $\eps^{\alpha\tilde{\eta}-1}$ involving the H\"older exponent $\tilde{\eta}$ of the coefficient $\tilde{\sigma}$ plays a key role.
The behavior of this term necessitates a separate analysis of the following two cases:
This dependency necessitates the separate analysis of the following two cases: Case 1: $\tilde{\eta} \in (1/\alpha, 1]$ and Case 2: $\tilde{\eta} = 1/\alpha$.

Case 1:
In this case, we have $\alpha\tilde{\eta}-1 > 0$. The optimization strategy is to choose $\delta$ as a constant and $\eps$ as a function of $B$ and $S$.
We make the simplest choice, $\delta=2$, and set $\eps$ to be a power of the dominant error term:
\begin{gather*}
p,q>0,\: \eps=\max\{B^p,S^q\}\textrm{~~and~~}\frac{C}{4} =\left(1+\frac{2^{2-\alpha}}{\log 2}\right)\hat{C}_{\alpha,T}.
\end{gather*}
We will specify $p$ and $q$ later in the analysis of Case 1.
Substituting these into the general estimate \eqref{alpha-1_norm_est}, and absorbing all $T,\alpha,K,k$-dependent constants into a generic constant $C$, the inequality simplifies to:
\begin{align*}
\sup_{0\le t\le T}\e\left[\left|X_t-\tilde{X}_t\right|^{\alpha-1}\right] \le  \frac{C}{4}\left(|x_0-\tilde{x}_0|^{\alpha-1}+\eps^{\alpha-1}+\frac{B}{\eps^{2-\alpha}}+\eps^{\alpha\tilde{\eta}-1}+\frac{S^\alpha}{\eps}\right).
\end{align*}
Our goal is to find the optimal exponents $p$ and $q$ that yield the best convergence rate as $B,S\to+0$.
To do this, we analyze the dominant terms in the two sub-cases.
To find the optimal convergence rate, we consider Subcase 1-1: $B^p<S^q$ and Subcase 1-2: $B^p\ge S^q$.

Subcase 1-1:
In this regime, the convergence rate is primarily determined by $B$.
The dominant terms in the error in the error $\eps$ that involve $B$ are of the order $O\left(B^{1-p(2-\alpha)}\right)$ and $O\left(B^{p(\alpha\tilde{\eta}-1)}\right)$ as we consider the limit $B\to +0$.
The overall convergence rate with respect to $B$ is determined by the slower of these two terms, specifically the one with the larger exponent in the limit $B\to+0$.
To optimize the rate, we balance the two exponents by setting them equal:
\begin{align*}
    1-p(2-\alpha)&=p(\alpha\tilde{\eta}-1),\\
    \implies p&=1/(\alpha\tilde{\eta}-\alpha+1).
\end{align*}

Subcase 1-2:
The dominant terms in the error $\eps$ involving $S$ are also of the order $O\left(S^{\alpha-q}\right)$ and $O\left(S^{q(\alpha\tilde{\eta}-1)}\right)$ since $S\in(0,1)$.
Similarly, for $q$, the solution $q=1/\tilde{\eta}$ is optimal when $\alpha-q=q(\alpha\tilde{\eta}-1)$.
By selecting the parameters $p$ and $q$ according to the analysis above, we can establish a more precise estimate.
The final convergence rate is determined by the slower of the two cases.
This leads to the following refined inequality:
\begin{align*}
\sup_{0\le t \le T}\e\left[\left|X_t-\tilde{X}_t\right|^{\alpha-1}\right] \le  C\left(\left|x_0-\tilde{x}_0\right|^{\alpha-1}+ \max\left\{B^{(\alpha\tilde{\eta}-1)/(\alpha\tilde{\eta}-\alpha+1)}, S^{\alpha-1/\tilde{\eta}}\right\}\right).
\end{align*}
This concludes the analysis for Case 1: $\tilde{\eta}\in(1/\alpha,1]$.

Case 2:
We now consider the boundary case where $\tilde{\eta}=1/\alpha$.
We adopt a different strategy where both $\eps$ and $\delta$ are chosen as powers of a single underlying parameter $\lambda$, which is itself a function of $B$ and $S$.
This allows for a more delicate balancing of all diverging terms simultaneously.
We set $p,q,r,s>0$, $C=C_{\alpha,T}$ and choose $\lambda=(\max\{B^p,S^q\})^{-1}$, $\eps =\lambda^{-r}$, $\delta=\lambda^{s(\alpha-1)}$.
Then we obtain the following estimate:
\begin{align*}
&\sup_{0\le t \le T}\e\left[\left|X_t-\tilde{X}_t\right|^{\alpha-1}\right]\\ 
    &\qquad\le
    C\bigg[|x_0-\tilde{x}_0|^{\alpha-1}+\lambda^{-r(\alpha-1)}+
        \left\{\lambda^{r(2-\alpha)}
        +
        \frac{\lambda^{r(2-\alpha)+s(\alpha-1)}\left(1-\lambda^{-s(\alpha-1)}\right)^{\alpha-1}}{s(\alpha-1)\log \lambda} \right\}B\\
        &\qquad\quad + \frac{\lambda^{(s-r)(\alpha-1)}(1-\lambda^{-s(\alpha-1)})^{\alpha-1}}{s(\alpha-1)\log \lambda} 
            +\frac{1}{s(\alpha-1)\log \lambda}+\frac{\lambda^{r+s(\alpha-1)}}{s(\alpha-1) \log \lambda}S^\alpha\bigg],\\
    &\qquad\le C\bigg[\left|x_0-\tilde{x}_0\right|^{\alpha-1} + \left(\log \lambda\right)^{-1}\bigg\{\frac{\log \lambda}{\lambda^{r(\alpha-1)}} + B \lambda^{r(2-\alpha)}\log \lambda + \frac{B\lambda^{r(2-\alpha)+s(\alpha-1)}\left(1-\lambda^{-s(\alpha-1)}\right)^{\alpha-1}}{s} \\
    &\qquad \quad+ \frac{\lambda^{(s-r)(\alpha-1)}\left(1-\lambda^{-s(\alpha-1)}\right)^{\alpha-1}}{s(\alpha-1)} + \frac{1}{s(\alpha-1)}+\frac{S^\alpha\lambda^{r+s(\alpha-1)}}{s(\alpha-1)} \bigg\}\bigg].
\end{align*}
For the fourth term within the curly braces to be bounded independently of $\lambda$ for any arbitrary $\lambda>1$, it is necessary that $s-r\ge 0$.
As can be seen from the other terms, in order for $p$ and $q$ to take large values, it is necessary to make $r$ and $s$ as small as possible.
Hence, it is best to set $r=s$.
We let $p=q=1/s$, since the condition imposed on $p$ and $q$ in this case is $p,q\le 1/s$.
For simplicity, let $\theta=\max\{B,S\}$.
Note that $\lambda=\theta^{-1/s}$, we have
\begin{align*}
\sup_{0\le t \le T}\e\left[\left|X_t-\tilde{X}_t\right|^{\alpha-1}\right]
    &\le C\bigg[\left|x_0-\tilde{x}_0\right|^{\alpha-1} + \left(\log \theta^{-1}\right)^{-1}\bigg\{-\theta^{\alpha-1}\log \theta - \frac{B \log \theta}{\theta^{2-\alpha}} + \frac{B \left(1-\theta^{\alpha-1}\right)^{\alpha-1}}{\theta}\\
    &\qquad + \frac{\left(1-\theta^{\alpha-1}\right)^{\alpha-1}}{\alpha-1} + \frac{1}{\alpha-1}+\frac{S^\alpha }{(\alpha-1)\theta^{\alpha}} \bigg\}\bigg],\\
    &\le C\bigg[\left|x_0-\tilde{x}_0\right|^{\alpha-1} + \left(\log \theta^{-1}\right)^{-1}\bigg\{\frac{1}{e(\alpha-1)} + \frac{1}{e(\alpha-1)} + 1 \\
    &\qquad + \frac{1}{\alpha-1} + \frac{1}{\alpha-1}+\frac{1}{\alpha-1} \bigg\}\bigg].
\end{align*}
This final inequality is obtained by using the fact that $\theta,B,S\in(0,1)$ to bound each term from above.
Therefore, defining $C'=C[1+2/\{e(\alpha-1)\}+3/(\alpha-1)]$ we arrive at the following inequality:

\begin{align*}
\sup_{0\le t \le T}\e\left[\left|X_t-\tilde{X}_t\right|^{\alpha-1}\right]
    &\le C'\left\{\left|x_0-\tilde{x}_0\right|^{\alpha-1} + \left(\log \frac{1}{\max\{B,S\}}\right)^{-1}\right\}.
\end{align*}
This inequality provides the explicit convergence rates for the $L^{\alpha-1}(\Omega)$ distance, as stated in Theorem \ref{mainthm}.
This concludes the proof.
\end{proof}

\section{Convergence in probability of the solution paths}\label{distanceprob}
Having established a stability estimate for the convergence in expectation, we now turn to a related but distinct question: the stability of solution paths in the sense of convergence in probability.
This type of convergence is often of practical interest as it concerns the probability of large deviations between two paths.
In this section, we leverage our previous result to establish a rate for the convergence of the time-supremum of the distance between two solutions.
The key idea is to show that our $L^{\alpha-1}(\Omega)$ estimate from Theorem \ref{mainthm} can be translated into a bound on the tail probability $\p\left(\sup_{t\in[0,T]} \left|X_t - \tilde{X}_t\right|^{\alpha-1} > h\right)$ by applying a maximal inequality for quasi-martingales.


The proof strategy is analogous to that employed in our prior work \cite{Nakagawa} and relies on the theory of quasi-martingales.
The key steps involve showing that the mollified process $\ude(Y)=(\ude(Y_t))_{0\le t\le T}$ is a quasi-martingale, where $Y_t = X_t - \tilde{X}_t$, and $\ude$ is the smooth approximation of $|x|^{\alpha-1}$ introduced in Lemma \ref{abstruct}.
We then apply a maximal inequality from Kurtz \cite{Kurtz} to this process.
This procedure allows us to translate the bounds for time-dependent perturbations obtained in Theorem \ref{mainthm} into a statement about tail probabilities.

\begin{theorem}\label{mainthm2}
Assume the same conditions as in Theorem \ref{mainthm}. Then, there exists a positive constant $C$ such that for any $h>0$,
\begin{align*}
\p\left(\sup_{0\le t \le T}\left|X_t-\tilde{X}_t\right|^{\alpha-1}>h\right) \le
\begin{cases}
    \frac{C}{h}\left\{\left|x_0-\tilde{x}_0\right|^{\alpha-1}+ \max\left\{B^{(\alpha\tilde{\eta}-1)/(\alpha\tilde{\eta}-\alpha+1)}, S^{\alpha-1/\tilde{\eta}}\right\}\right\}& \text{if~~}\tilde{\eta} \in(1/\alpha,1]\\
	\frac{C}{h}\left[\left|x_0-\tilde{x}_0\right|^{\alpha-1}+ \left\{\log\left(\frac{1}{\max\left\{B, S\right\}}\right)\right\}^{-1}\right] & \textrm{if~~}\tilde{\eta} =1/\alpha.
\end{cases}
\end{align*}
\end{theorem}

\subsection{Proof strategy for Theorem \ref{mainthm2}}
The proof of Theorem \ref{mainthm2} follows a strategy analogous to that used for the zero-drift case in \cite{Nakagawa}.
The core of the argument is to apply a maximal inequality for quasi-martingales to the mollified process $\ude(Y)=\left(\ude\left(Y_t\right)\right)_{0\le t\le T}$, where $Y_t \coloneqq X_t - \tilde{X}_t$ is the difference between the two solutions.
Let us briefly recall the main tool. 
A c\`adl\`ag adapted process $Q$ is called a quasi-martingale if its mean variation, $V_T(Q)$, is finite. The utility of this concept stems from the following inequality by Kurtz \cite{Kurtz}.

\begin{Lem}(\cite[Lemma 5.3]{Kurtz})\label{quasiprop}
Let $Q$ be a quasi-martingale on $[0,T]$. Then, for any $h>0$:
\begin{align*}
h\p\left(\sup_{t\in[0,T]}|Q_t| > h\right) \le V_T(Q) + \e[|Q_T|].
\end{align*}
\end{Lem}
To apply this lemma, we first need to bound the mean variation $V_T(u_{\delta,\varepsilon}(Y))$.
As shown in the proof of Theorem \ref{mainthm}, the process $u_{\delta,\varepsilon}(Y)$ can be decomposed as
\begin{align*}
    \ude(Y_t) = \ude(Y_0) + M_t^{\delta,\eps} + I_t^{\delta,\eps,b} + I_t^{\delta,\eps,\sigma}.
\end{align*}
The mean variation is therefore bounded by the sum of the absolute values of the conditional expectations of the increments of each term.
Since $M_t^{\delta,\eps}$ is a true martingale, its contribution to the mean variation is zero.
The estimation of the term $I_t^{\delta,\eps,\sigma}$ is identical to the analysis in \cite{Nakagawa}.
The only new component in this analysis is the contribution from the drift term, $I_t^{\delta,\eps,b}$.
The bounds for this term were already derived in \eqref{Itb}, \eqref{Itb_term1} and \eqref{Itb_term2} of the present paper.
By combining these estimates, we can bound $V_T(\ude(Y))$ and $\e[|\ude(Y_T)|]$.
Applying Lemma \ref{quasiprop}, the inequality \eqref{alpha-1_norm_est} and \eqref{abstructYt} then yields the desired result.
Since the remaining steps are identical to those in proof of Theorem \ref{mainthm}, we omit the detailed calculations.

\section{Appendices}\label{Appendices}
This appendix serves two main purposes.
First, we explore the implications of our main stability result (Theorem \ref{mainthm}) in two important special cases.
We begin by examining the time-homogeneous case (Section \ref{time-homogeneous_case}), where the perturbed coefficients $\tilde{b}$ and $\tilde{\sigma}$ are also independent of time.
This simplifies the coefficient distance measures $B$ and $S$ into more intuitive weighted $L^p$ norms. Subsequently, we consider the case where the coefficient distance is measured by the supremum norm (Section \ref{supnorm_case}).
This alternative framework allows for a significant relaxation of the continuity and boundedness assumptions on the coefficients.
Second, we provide the detailed proofs for several technical lemmas that were stated without proof in the main body of the paper to maintain the flow of the argument.

\subsection{The time-homogeneous case}\label{time-homogeneous_case}
We now present an important special case of our main result, Theorem \ref{mainthm}, where all coefficients are time-homogeneous.
This setting serves three crucial purposes.
First, it illustrates our main theorem in a more concrete and intuitive context.
Second, it clarifies the meaning of the abstract distance measures $B$ and $S$, showing that they naturally reduce to the weighted $L^p$-norms.
Finally, this corollary provides the foundational stability estimate for analyzing the convergence of approximating sequences of SDEs, which will be the primary application demonstrated in Section \ref{Construct}.

\begin{Cor}\label{Cor_time-homogeneous}
Let $Z$ be a one-dimensional symmetric $\alpha$-stable process with $\alpha \in (1,2)$. Consider two processes, $X$ and $\tilde{X}$, satisfying the SDEs
\begin{align*}
X_t = x_0 + \int_0^t b(X_s) \rd s + \int_0^t \sigma(X_{s-}) \rd Z_s \quad \text{and} \quad \tilde{X}_t = \tilde{x}_0 + \int_0^t \tilde{b}(\tilde{X}_s) \rd s +  \int_0^t \tilde{\sigma}(\tilde{X}_{s-}) \rd Z_s,
\end{align*}
for $t\in[0,T]$ with initial values $x_0, \tilde{x}_0\in\real$.

Suppose the coefficients satisfy the following conditions for some constants $K, k, K'>0$, and $\tilde{\eta} \in [1/\alpha,1]$:
\begin{itemize}
    \item[(i)'] The drift coefficient $b$ is continuous bounded and $\tilde{b}$ is Lipschitz continuous:
    \begin{align*}
       |b(x)|\le K, \quad |\tilde{b}(x)-\tilde{b}(y)|\le K |x-y|. 
    \end{align*}
    \item[(ii)'] The jump coefficients $\sigma$ and $\tilde{\sigma}$ are bounded and H\"older continuous, and $\sigma$ is strictly positive:
    \begin{align*}
       k \le \sigma(x)\le K, \quad |\sigma^\alpha(x)-\sigma^\alpha(y)|\le K|x-y|^\eta,\\
       |\tilde{\sigma}(x)|\le K', \quad |\tilde{\sigma}(x)-\tilde{\sigma}(y)|\le K|x-y|^{\tilde{\eta}}.
    \end{align*}
    \item[(iii)'] The distances between the coefficients are small in the weighted norms defined in \eqref{weighted_norm}:
    \begin{align*}
       B \coloneqq \|b-\tilde{b}\|_{L^1\left(\mu_{x_0,\sigma}^{\alpha,T}\right)} < 1, \quad S \coloneqq \|\sigma-\tilde{\sigma}\|_{L^\alpha\left(\mu_{x_0,\sigma}^{\alpha,T}\right)} < 1. 
    \end{align*}
\end{itemize}
Then, there exists a positive constant $C = C(T, \alpha, K, k)$ such that\\
\noindent (1) Convergence in Expectation:
\begin{align*}
\sup_{t \in [0,T]} \e\left[|X_t-\tilde{X}_t|^{\alpha-1}\right] \le 
 \begin{cases}
    C\left\{\left|x_0-\tilde{x}_0\right|^{\alpha-1}+ \max\left\{B^{(\alpha\tilde{\eta}-1)/(\alpha\tilde{\eta}-\alpha+1)}, S^{\alpha-1/\tilde{\eta}}\right\}\right\}& \textrm{if~~}\tilde{\eta} \in(1/\alpha,1]\\
	C\left[\left|x_0-\tilde{x}_0\right|^{\alpha-1}+ \left\{\log\left(\frac{1}{\max\left\{B, S\right\}}\right)\right\}^{-1}\right] & \textrm{if~~}\tilde{\eta} =1/\alpha,
  \end{cases}
\end{align*}
\noindent (2) Convergence in Probability: For any $h>0$,
\begin{align*}
\p\left(\sup_{t \in [0,T]}|X_t-\tilde{X}_t|^{\alpha-1}>h\right) \le
 \begin{cases}
    \frac{C}{h}\left\{\left|x_0-\tilde{x}_0\right|^{\alpha-1}+ \max\left\{B^{(\alpha\tilde{\eta}-1)/(\alpha\tilde{\eta}-\alpha+1)}, S^{\alpha-1/\tilde{\eta}}\right\}\right\}& \text{if~~}\tilde{\eta} \in(1/\alpha,1]\\
	\frac{C}{h}\left[\left|x_0-\tilde{x}_0\right|^{\alpha-1}+ \left\{\log\left(\frac{1}{\max\left\{B, S\right\}}\right)\right\}^{-1}\right] & \textrm{if~~}\tilde{\eta} =1/\alpha.
 \end{cases}
\end{align*}
\end{Cor}
\begin{proof}
This result follows from an application of Theorem \ref{mainthm} to the time-homogeneous setting.
The main task is to show that the general coefficient distance measures $B$ and $S$ from Theorem \ref{mainthm} can be bounded by the weighted norms stated in assumption (iii)' of this corollary.
Let us consider the drift term $B$.
By applying the two-sided estimate for the transition density $p_s$ from inequality \eqref{SDE_density_estimate}, we have
\begin{align*}
    \int_0^T \int_\real |b(y)-\tilde{b}(y)|p_s(x_0,y) \rd s \rd y \le M \left\|b-\tilde{b}\right\|_{L^1\left(\mu_{x_0,\sigma}^{\alpha,T}\right)}.
\end{align*}
where the measure $\mu_{x_0,\sigma}^{\alpha,T}$ is defined via the principal part of the transition density.
The term $S$ follows analogously.
The significance of this weighted norm lies in the properties of the transition density $p_t^0(x_0,\cdot)$.
As noted in Section \ref{SDE_density_section}, the density $p_t^0(x_0,\cdot)$ is constructed from $g^{(\alpha)}$, which has heavy tails decaying at a rate of $(|\cdot|\vee 1)^{-1-\alpha}$.
This decay property is inherited by our measure $\mu_{x_0,\sigma}^{\alpha,T}$, ensuring that it gives much less weight to points far from the initial value $x_0$.
This effectively localizes the error measure to the region where the process is most likely to be found, making the norm particularly well-suited for analyzing stochastic processes.
With the connection between $B, S$ and the weighted norms established, the main stability estimate from Theorem \ref{mainthm} directly applies. 


\end{proof}

\subsection{An alternative framework: The supremum norm case}\label{supnorm_case}
We now investigate the stability problem under a different set of assumptions, where the distances between coefficients are measured using the supremum norm.
This alternative framework is motivated by the fact that it allows for a significant relaxation of the continuity and boundedness assumptions imposed on the coefficients in Theorem \ref{mainthm}.
While this approach is less general in its handling of time-dependence compared to our main theorem, it provides a powerful result under different conditions.
\begin{theorem}\label{mainthm_sup1}
Let $Z$ be a one-dimensional symmetric $\alpha$-stable process with $\alpha \in (1,2)$. Consider two processes, $X$ and $\tilde{X}$, satisfying the SDEs
\begin{align*}
X_t = x_0 + \int_0^t b(s,X_s) \rd s + \int_0^t \sigma(s,X_{s-}) \rd Z_s \quad \text{and} \quad \tilde{X}_t = \tilde{x}_0 + \int_0^t \tilde{b}(s,\tilde{X}_s) \rd s +  \int_0^t \tilde{\sigma}(s,\tilde{X}_{s-}) \rd Z_s,
\end{align*}
for $t\in[0,T]$ with initial values $x_0, \tilde{x}_0\in\real$.
Let the framework be the same as in Theorem \ref{mainthm}.
However, instead of the assumptions (i)-(iii) therein, suppose the coefficients satisfy the following conditions for some constant $\tilde{\eta} \in [1/\alpha,1]$:
\begin{itemize}
    \item[(a)] The coefficient $\tilde{b}$ is Lipschitz continuous and $\tilde{\sigma}$ is $\tilde{\eta}$-H\"older continuous in space: there exist nonnegative functions $f_{\tilde{b}}$ and $f_{\tilde{\sigma}}$ such that
    \begin{align*}
       |\tilde{b}(t,x)-\tilde{b}(t,y)| \le f_{\tilde{b}}(t) |x-y|
       ,\ |\tilde{\sigma}(t,x)-\tilde{\sigma}(t,y)| \le f_{\tilde{\sigma}}(t)|x-y|^{\tilde{\eta}}
       ,\ \int_0^T f_{\tilde{b}}(s)\rd s<\infty
       ,\ \int_0^T f_{\tilde{\sigma}}^\alpha(s)\rd s<\infty.
    \end{align*}
    \item[(b)] The differences between coefficients are small in the supremum norm.
    Let the quantities $B_\infty$ and $S_\infty$ be defined as follows.
    For the drift coefficients, $B_\infty$ is defined by either
    \begin{align*}
        B_\infty := \int_0^T \left\|b(s,\cdot) - \tilde{b}(s,\cdot)\right\|_\infty \rd s \quad \text{or} \quad B_\infty := \sup_{t\in[0,T]} \left\|b(t,\cdot) - \tilde{b}(t,\cdot)\right\|_\infty.
    \end{align*}
    For the jump coefficients, $S_\infty$ is defined by either
    \begin{align*}
        S_\infty := \left(\int_0^T \left\|\sigma(s,\cdot) - \tilde{\sigma}(s,\cdot)\right\|_\infty^\alpha \rd s \right)^{1/\alpha} \quad \text{or} \quad S_\infty := \sup_{t\in[0,T]} \left\|\sigma(t,\cdot) - \tilde{\sigma}(t,\cdot)\right\|_\infty.
    \end{align*}
    The theorem's conclusions hold provided that there exists a choice of definitions for $B_\infty$ and $S_\infty$ from the options above such that $B_\infty<1$ and $S_\infty<1$.
    \item[(c)] Additionally, the squared difference of the jump coefficients is integrable:
    \begin{align*}
        \int_0^T \left\|\sigma(s,\cdot)-\tilde{\sigma}(s,\cdot)\right\|^2_\infty\rd s<\infty.
    \end{align*}
    These can be set as either $B_\infty$ or $S_\infty$.
\end{itemize}
Then, under these assumptions, the conditions on the coefficients can be significantly relaxed compared to Theorem \ref{mainthm}.
Notably:
\begin{itemize}
    \item The coefficients $b$ and $\sigma$ of the baseline process $X$ are no longer required to be continuous or bounded; mere measurability is sufficient. The only condition required for these coefficients is that a weak solution of $X$ exists.
    \item The coefficients $\tilde{b}$ and $\tilde{\sigma}$ of the perturbed process $\tilde{X}$ are no longer required to be bounded. Note that $\tilde{X}$ is assumed to have a strong solution.
    \item One may define $B_\infty$ using either the $L^1([0,T]; L^\infty)$ or $L^\infty([0,T]; L^\infty)$ norm of $b-\tilde{b}$. The former typically provides a tighter smallness condition for short time intervals, whereas the latter is preferable for a long time horizon where the integral norm might exceed unity. The same applies to $S_\infty$.
\end{itemize}
There exists a positive constant $C = C\left(T, \alpha, \int_0^T f_{\tilde{b}}(s)\rd s, \int_0^T f_{\tilde{\sigma}}(s)\rd s\right)$ such that the following stability estimates hold:

\noindent (1) Convergence in expectation:
\begin{align*}
\sup_{t \in [0,T]} \e\left[|X_t-\tilde{X}_t|^{\alpha-1}\right] \le 
 \begin{cases}
    C\left\{\left|x_0-\tilde{x}_0\right|^{\alpha-1}+ \max\left\{B_\infty^{(\alpha\tilde{\eta}-1)/(\alpha\tilde{\eta}-\alpha+1)}, S_\infty^{\alpha-1/\tilde{\eta}}\right\}\right\}& \textrm{if~~}\tilde{\eta} \in(1/\alpha,1]\\
	C\left[\left|x_0-\tilde{x}_0\right|^{\alpha-1}+ \left\{\log\left(\frac{1}{\max\left\{B_\infty, S_\infty\right\}}\right)\right\}^{-1}\right] & \textrm{if~~}\tilde{\eta} =1/\alpha,
  \end{cases}
\end{align*}
\noindent (2) Convergence in probability: For any $h>0$,
\begin{align*}
\p\left(\sup_{t \in [0,T]}|X_t-\tilde{X}_t|^{\alpha-1}>h\right) \le
 \begin{cases}
    \frac{C}{h}\left\{\left|x_0-\tilde{x}_0\right|^{\alpha-1}+ \max\left\{B_\infty^{(\alpha\tilde{\eta}-1)/(\alpha\tilde{\eta}-\alpha+1)}, S_\infty^{\alpha-1/\tilde{\eta}}\right\}\right\}& \text{if~~}\tilde{\eta} \in(1/\alpha,1]\\
	\frac{C}{h}\left[\left|x_0-\tilde{x}_0\right|^{\alpha-1}+ \left\{\log\left(\frac{1}{\max\left\{B_\infty, S_\infty\right\}}\right)\right\}^{-1}\right] & \textrm{if~~}\tilde{\eta} =1/\alpha.
 \end{cases}
\end{align*}
\end{theorem}

\begin{proof}[Proof Sketch for Theorem \ref{mainthm_sup1}]
The proof follows the exact same logical steps as the proofs of Theorems \ref{mainthm} and \ref{mainthm2}: applying Ito's formula, estimating each term, using Gronwall's inequality, and optimizing the parameters.
The only modification lies in how the integral terms involving coefficient differences are bounded.
Specifically, the integral term $\int_0^T \e[|b(s,X_s)-\tilde{b}(s,X_s)|] ds$ in inequality \eqref{alpha-1_est} is no longer connected to a weighted norm.
Instead, it is bounded directly using the definition of $B_\infty$:
\begin{align*}
   \int_0^T \e[|b(s,X_s)-\tilde{b}(s,X_s)|] ds \le
   \begin{cases}
   \int_0^T \left\|b(s,\cdot)-\tilde{b}(s,\cdot)\right\|_\infty \rd s & \textrm{if~} B_\infty = \int_0^T \left\|b(s,\cdot)-\tilde{b}(s,\cdot)\right\|_\infty \rd s,\\
   T\sup_{t\in[0,T]}\|b(t,\cdot)-\tilde{b}(t,\cdot)\|_\infty \rd s & \textrm{if~} B_\infty = \sup_{t\in[0,T]}\left\|b(t,\cdot)-\tilde{b}(t,\cdot)\right\|_\infty.
   \end{cases}
\end{align*}
A similar direct bound is used for the term involving $S_\infty$.
This modification is the key to the main advantage of this theorem: since the estimation does not require the transition density $p_t(x_0,\cdot)$, the assumptions on the coefficients $b$ and $\sigma$ of the baseline process $X$ can be relaxed to mere measurability.
The martingale property of $M^{\delta,\eps}$ is also guaranteed under these assumptions, as the boundedness of $\int_0^T \left\|\sigma(s,\cdot)-\tilde{\sigma}(s,\cdot)\right\|^2_\infty\rd s$ is sufficient for local integrability (see Subsection \ref{M-martingale}).
The rest of the argument, including the parameter optimization, proceeds identically. 
\end{proof}

\subsection{Proof of the martingale property for $M^{\delta,\eps}$}\label{M-martingale}
This section is dedicated to proving that the process $M^{\delta,\eps}$, defined in \eqref{udeYt}, is a true martingale.
This property is crucial for the main proofs of Theorem \ref{mainthm}, \ref{mainthm2} and \ref{mainthm_sup1} and Corollary \ref{Cor_time-homogeneous}, as it allows us to eliminate the stochastic integral term when taking expectations.
The proof, which follows the argument in \cite[Subsection 5.1]{Nakagawa}, hinges on the properties of the mollified function $\ude$ introduced in Lemma \ref{abstruct}.
Specifically, we need to establish that $\ude$ is Lipschitz continuous, which is a consequence of its first derivative being bounded.
The following lemma summarizes the necessary properties.

\begin{Lem}\label{udelip}
The function $\ude$ is of class $C^2$.
Its first derivative, $\ude'$, satisfies the bound given in inequality \eqref{ude'ineq}, which implies that $\ude'$ is bounded for any fixed $\eps>0$ and $\delta>1$.
\end{Lem}
The proof of this lemma involves a straightforward case-by-case analysis.
We provide the details below.
\begin{proof}
The smoothness of $\ude$ follows from standard properties of convolutions.
Since $u(x)=|x|^{\alpha-1}$ is locally integrable and the mollifier $\pde$ is a smooth function with compact support (in fact, a Schwartz function), their convolution $\ude = u \ast \pde$ is infinitely differentiable.
The main task is to prove the bound on the first derivative, $\ude'$, stated in inequality \eqref{ude'ineq}.
Differentiating under the integral sign, we have:
\begin{align*}
    \ude'(x) &= (\alpha-1)\int_{\real} \mathrm{sgn}(x-y)|x-y|^{\alpha-2}\psi_{\delta,\varepsilon}(y)\rd y.
\end{align*}
Since $\pde$ has its support on the compact interval $\left[\eps\delta^{-1}, \eps\right]$, we can restrict the domain of integration.
By taking the absolute value, we obtain the initial bound:
\begin{align*}
    |\ude'(x)| \le (\alpha-1)\int_{\eps \delta^{-1}}^\eps |x-y|^{\alpha-2}\pde(y)\rd y.
\end{align*}
To derive the bound \eqref{ude'ineq} from this integral representation, we analyze the behavior of the integrand $|x-y|^{\alpha-2}$.
The estimate depends on the distance between $x$ and the support of $\pde$, which is $[\eps\delta^{-1,\eps]}$.
We therefore proceed with a case-by-case analysis based on the position of $x$: Case 1 ($x > 2\eps$), Case 2 ($x < -2\eps$), and Case 3 ($|x| \le 2\eps$).

Case 1:
In this regime, for any $y\in[\eps\delta^{-1},\eps]$, the difference $x-y>\eps>0$.
Since $\alpha\in(1,2)$, the function $z\mapsto z^{\alpha-2}$ is decreasing for $z>0$.
Therefore, we can bound the term $|x-y|^{\alpha-2}$
\begin{align*}
|x-y|^{\alpha-2}=(x-y)^{\alpha-2}\le (x-\eps)^{\alpha-2}.
\end{align*}
Since the condition $2\eps<x$ implies $x-\eps>x-x/2=x/2$, we obtain
\begin{align*}
    (x-\eps)^{\alpha-2} < (x/2)^{\alpha-2} = 2^{2-\alpha} x^{\alpha-2} = 2^{2-\alpha} |x|^{\alpha-2}
\end{align*}
Substituting this uniform bound into the integral for $|\ude'(x)|$ and using the fact that $\pde$ integrates to one, we obtain the desired estimate for this case:
\begin{align*}
     \left|\ude'(x)\right|\le 2^{2-\alpha}(\alpha-1)|x|^{\alpha-2}.
\end{align*}

Case 2:
For any $y\in[\eps\delta^{-1},\eps]$, the term $x-y<-2\eps-\eps\delta^{-1}<0$.
Since $|x-y|^{\alpha-2}$ achieves its maximum value $(|x|+\eps\delta^{-1})^{\alpha-2}$ at $y=\eps\delta^{-1}$,
\begin{align*}
    |\ude'(x)|
    &\le (\alpha-1)(|x|+\eps\delta^{-1})^{\alpha-2}\\
    &\le (\alpha-1)|x|^{\alpha-2}
\end{align*}
Finally, since $\alpha \in (1,2)$ implies $2^{2-\alpha} > 1$, the desired inequality $|u'_{\delta,\varepsilon}(x)| \le 2^{2-\alpha}(\alpha-1)|x|^{\alpha-2}$ is satisfied in this case as well.

Case 3:
In this final case, we bound the integral directly. We utilize the explicit upper bound for the mollifier, $\pde(y) \in\left[0,\frac{2}{y\log\delta}\right]$, which holds on its support $\left[\eps\delta^{-1}, \eps\right]$.
This gives:
\begin{align*}
    |\ude'(x)|
    \le (\alpha-1) \frac{2\delta}{\eps \log \delta} \int_{\eps \delta^{-1}}^\eps |x-y|^{\alpha-2}\rd y
\end{align*}
The integral term $\int |x-y|^{\alpha-2} dy$ depends on $x$.
To obtain a uniform bound for $|x| \le 2\eps$, we find the maximum value of this integral.
The function $x \mapsto \int_a^b |x-y|^{\alpha-2} dy$ is a convex function, and for a symmetric interval around the origin, its maximum would be at the endpoints.
In our case, the integral is maximized when $x$ is at the center of the integration interval, i.e., $x = (\eps\delta^{-1}+\eps)/2$.
By substituting this maximizing value of $x$, we can compute the upper bound for the integral:
\begin{align*}
    |\ude'(x)|
    &\le (\alpha-1) \frac{2\delta}{\eps \log \delta} \int_{\eps \delta^{-1}}^\eps \left|\frac{\eps+\eps\delta^{-1}}{2}-y\right|^{\alpha-2}\rd y,\\
    &=\frac{2^{3-\alpha} \delta(1-\delta^{-1})^{\alpha-1}}{\eps^{2-\alpha}\log \delta}.
\end{align*}
This bound holds for all $|x| \le 2\varepsilon$.
Combining the results from all three cases completes the proof of inequality \eqref{ude'ineq}.

Finally, to justify the well-definedness of $L_\alpha \ude$ in Section \ref{Lude_well-defined}, we must also show that the second derivative, $\ude''$, is bounded.
Since $\ude$ is $C^\infty$, we can again differentiate under the integral sign twice.
This yields $\ude'' = u' \ast \pde'$, which can be written explicitly as:
\begin{align*}
    \ude''(x) = (\alpha-1) \int_\real \mathrm{sgn}(x-y)|x-y|^{\alpha-2} \pde'(y) \rd y.
\end{align*}
The support of the derivative $\pde'$ is also contained within $\left[\eps\delta^{-1}, \eps\right]$.
Furthermore, as $\pde$ is a $C_c^\infty$ function, its derivative $\pde'$ is bounded, i.e., $\|\pde'\|_\infty < \infty$.
By taking the absolute value and restricting the integration domain to the support of $\pde'$, we obtain the following bound:
\begin{align*}
    |\ude''(x)| 
    &\le (\alpha-1) \int_{\eps\delta^{-1}}^{\eps} |x-y|^{\alpha-2} |\pde'(y)| dy \\
    &\le (\alpha-1) \|\pde'\|_\infty \int_{\eps\delta^{-1}}^{\eps} |x-y|^{\alpha-2} \rd y.
\end{align*}
The remaining integral, $\int_{\eps\delta^{-1}}^{\eps} |x-y|^{\alpha-2} dy$, is finite for any $x \in \real$ because the singularity of the integrand at $y=x$ is of order $\alpha-2 > -1$, making it integrable.
As we saw in the analysis of Case 3 above, this integral attains its maximum for $x$ within or near the interval $\left[\eps\delta^{-1}, \eps\right]$, and this maximum value is finite.
Therefore, $|\ude''|$ is uniformly bounded for all $x \in \real$.
This completes the proof of Lemma \ref{udelip}.
We have established that both the first and second derivatives of $\ude$ are uniformly bounded, properties that are essential for the arguments that follow.
\end{proof}

With the properties of $\ude$ established, we now return to the proof that $M^{\delta,\eps}$ is a martingale.
Our strategy is to decompose $M^{\delta,\eps}$ into its large-jump and small-jump components, and show that each component is a martingale.
\begin{align*}
M_t^{\delta,\eps,1}&=\int_{0}^{t}\int_{|z|\ge 1}\{ \ude\left(Y_{s-}+\left(\sigma(X_{s-})-\tilde{\sigma}\left(\tilde{X}_{s-}\right)\right)z\right)-\ude\left(Y_{s-}\right)\}\tilde{N}(\rd z,\rd s),\\
M_t^{\delta,\eps,2}&=\int_{0}^{t}\int_{|z|< 1}\{ \ude\left(Y_{s-}+\left(\sigma(X_{s-})-\tilde{\sigma}\left(\tilde{X}_{s-}\right)\right)z\right)-\ude\left(Y_{s-}\right)\}\tilde{N}(\rd z,\rd s).
\end{align*}
Clearly, we have $M_t^{\delta,\eps} = M_t^{\delta,\eps,1} + M_t^{\delta,\eps,2}$.
We will now show that both $\left(M_t^{\delta,\eps,1}\right)_{0\le t\le T}$ and $\left(M_t^{\delta,\eps,2}\right)_{0\le t\le T}$ are martingales.

First, we consider the small-jump component, $\left(M_t^{\delta,\eps,2}\right)_{0\le t\le T}$.
We will show it is a square-integrable martingale by verifying that the expectation of its quadratic variation is finite.
Using the Lipschitz continuity of $\ude$ (from Lemma \ref{abstruct}) and the boundedness of the coefficients $\sigma$ and $\tilde{\sigma}$ (as assumed in Theorems \ref{mainthm} and \ref{mainthm2}), we can bound the expected quadratic variation as follows:
\begin{align*}
    \e \left[ \left\langle M^{\delta,\eps,2} \right\rangle_t \right]
    &=\e \left[\int_{0}^{t}\int_{|z|<1} \left|\ude\left(Y_s+\left(\sigxs\right)z\right)-\ude\left(Y_{s-}\right)\right|^2\nu(\rd z)\rd s\right]\\
    &\le \|\ude'\|_{\infty}^2 \e \left[\int_{0}^{t}\int_{|z|<1}\left|\left(\sigxs\right)z\right|^2\nu(\rd z)\rd s\right]\\
    &\le \|u'_{\delta,\eps}\|_{\infty}^2 \left(K^2 t + \int_0^t g_{\tilde{\sigma}}^2(s)\rd s\right) \int_{|z|<1} |z|^2 \nu(\rd z).
\end{align*}
This is finite because the integral $\int_{|z|<1} |z|^2 \nu(\rd z) = \int_{|z|<1} c_\alpha |z|^{1-\alpha} \rd z$ converges for $\alpha < 2$.
The same conclusion holds under the supremum norm assumption of Theorem \ref{mainthm_sup1}, where $K^2 t + \int_0^t g_{\tilde{\sigma}}^2(s)\rd s$ is replaced by a constant related to $\int_0^t \left\|\sigma(s,\cdot)-\tilde{\sigma}(s,\cdot)\right\|^2_\infty\rd s$.
Therefore, $\left(M_t^{\delta,\eps,2}\right)_{0\le t\le T}$ is a true $L^2$-martingale (see, e.g., \cite[Theorem 4.2.3]{Applebaum}, ).

Next, we address the large-jump component, $\left(M_t^{\delta,\eps,1}\right)_{0\le t\le T}$.
We show it is an $L^1$-martingale by demonstrating that the expected total variation of its jumps is finite.
The argument is similar:
\begin{align*}
    &\e \left[\int_{0}^{t}\int_{|z|<1} \left|\ude\left(Y_s+\left(\sigxs\right)z\right)-\ude\left(Y_{s-}\right)\right|\nu(\rd z)\rd s\right]\\
    &\le \|\ude\|_\infty \e \left[\int_0^t \int_{|z|\ge 1} \left|\left(\sigxs\right)z\right| \nu(\rd z)ds \right] \\
    &\le C \int_{|z|\ge 1} |z| \nu(\rd z).
\end{align*}
This final integral is finite because $\int_{|z|\ge 1} |z| \nu(\rd z) = \int_{|z|\ge 1} c_\alpha |z|^{-\alpha} \rd z$ converges for $\alpha > 1$.
Thus, $\left(M_t^{\delta,\eps,1}\right)_{0\le t\le T}$ is an $L^1(\Omega)$-martingale (see, e.g., \cite[P.231]{Applebaum}).

Since $M^{\delta,\eps}$ is the sum of two true martingales, we conclude that it is also a martingale.

\subsection{Regarding the well-definedness of a function $L_\alpha \ude$}\label{Lude_well-defined}
We now justify that the expression $L_\alpha \ude$ is well-defined, even though $\ude$ is not a rapidly decreasing function.
This requires showing that the defining integral for $L_\alpha \ude(x)$ converges for any $x\in\real$.
We split the integral into two domains: $|y| > 1$ and $|y| \le 1$.
For the integral over $|y| > 1$, the integrand behaves like $(\ude(x+y)-\ude(x))/|y|^{1+\alpha}$.
Since $\ude$ is Lipschitz continuous (a consequence of the boundedness of $\ude'$ established in Lemma \ref{udelip}), the numerator is bounded by $C|y|$, ensuring the integral converges for $\alpha>1$.
For the integral over $|y|\le 1$, we use a second-order Taylor expansion of the integrand around $y=0$.
The leading terms cancel, and the remainder is of the order $y^2 \ude''(x+\theta y)$ for some $\theta\in(0,1)$.
The integral thus behaves like $\int_{|y|\le 1} y^2/|y|^{1+\alpha}\rd y$.
This converges because $\ude''$ is uniformly bounded, as shown in Lemma \ref{udelip}, and the exponent $2-(1+\alpha)=1-\alpha$ is greater than $-1$.
Since the integral converges on both domains, we conclude that $L_\alpha \ude$ is well-defined.
%

\subsection{Application to time-homogeneous approximations}\label{Construct}
In this section, we demonstrate a key application of our stability results for the time-homogeneous case, namely Corollary \ref{time-homogeneous_case}.
We prove the convergence of a sequence of solutions to SDEs \eqref{SDEn}, a problem central to numerical analysis and approximation theory.
We will show that if a sequence of time-homogeneous coefficients $(b_n, \sigma_n)$ forms a Cauchy sequence in the weighted norms, then the corresponding sequence of solutions $X^{(n)}$ converges to a process $X^\infty$, which is itself the solution to the SDE with the limiting coefficients.


Let us formalize the assumptions for this section.
We formalize the conditions for this section as follows.

\noindent\textbf{Assumption A.} \textit{Let $\left(x_0^{(n)}\right)_{n\in\n}$ be a sequence of initial values, and $(b_n)_{n\in\n}$ and $(\sigma_n)_{n\in\n}$ be sequences of time-homogeneous drift and jump coefficients, respectively.}
\begin{itemize}
    \item[\textit{(i)}] \textit{The initial values converge: $\lim_{n\to\infty} x_0^{(n)} = x_0$.}
    \item[\textit{(ii)}] \textit{The coefficient sequences $(b_n)_{n\in\n}$ and $(\sigma_n)_{n\in\n}$ are Cauchy sequences in the norms $L^1\left(\mu_{x_0,\sigma}^{\alpha,T}\right)$ and $L^\alpha\left(\mu_{x_0,\sigma}^{\alpha,T}\right)$, respectively.}
    \item[\textit{(iii)}] \textit{The coefficients are uniformly bounded and satisfy uniform continuity conditions. That is, there exist constants $K>0$ and $\eta \in [1/\alpha,1]$ such that for all $n \in \n$ and $x,y \in \real$:
    \begin{align}
        \sup_{n\in\n}\sup_{x\in\real}|b_n(x)| &< \infty, \quad \sup_{n\in\n} |b_n(x)-b_n(y)| \le K|x-y|, \label{Uni_Lip_b_n} \\
        0 < \inf_{n\in\n}\inf_{x\in\real}\sigma_n(x), \quad &\sup_{n\in\n}\sup_{x\in\real}\sigma_n(x) < \infty, \notag\\
        \sup_{n\in\n}|\sigma_n(x)-\sigma_n(y)| &\le K|x-y|^{\eta}. \label{Uni_hol_sigma_n}
    \end{align}}
\end{itemize}
We note that the uniform bounds and H\"older continuity on $\sigma_n$ imply the same for $\sigma_n^\alpha$, independently of $n$.

Our proof of the convergence of the sequence $\left(X^{(n)}\right)_{n\in\n}$ proceeds in two main steps. First, we establish that the sequence of processes is uniformly integrable, which is a key condition for exchanging limits and expectations.

\begin{Lem}\label{Uni_integ_Xn}
    Under Assumption A, for any $p\in(1,\alpha)$, the sequence of random variables $\left(\sup_{t\in[0,T]}\left|X_t^{(n)}\right|\right)_{n\in\n}$ is bounded in $L^p(\Omega)$ uniformly in $n\in\n$.
\end{Lem}
\begin{proof}
    We start by applying the inequality $|x+y+z|^p \le 3^{p-1}(|x|^p+|y|^p+|z|^p)$, which holds for any $x,y,z\in\real$ and $p\ge 1$.
    This gives us
\begin{align*}
    \e\left[\sup_{0\le t\le T}\left|X_t^{(n)}\right|^p\right]\le 3^{p-1} |x_0^{(n)}|^p +3^{p-1}\e\left[\sup_{0\le t\le T}\left|\int_0^t b_n\left(X_s^{(n)}\right)\rd s\right|^p\right] + 3^{p-1}\e\left[\sup_{0\le t\le T}\left|\int_0^t \sigma_n\left(X_{s-}^{(n)}\right)\rd Z_s\right|^p\right].
\end{align*}
We now bound each term on the right-hand side uniformly in $n\in\n$.
First, the sequence of initial values $\left(x_0^{(n)}\right)_{n\in\n}$ converges by Assumption A(i), so it is bounded.
For the drift term, we use the uniform boundedness of the coefficients from Assumption A(iii):
\begin{align*}
    \e\left[\sup_{0\le t\le T}\left|\int_0^t b_n\left(X_s^{(n)}\right)\rd s\right|^p\right]
    &\le
    \sup_{n\in\n}\sup_{x\in\real}\|b_n\|^p_\infty T^p <\infty.
\end{align*}
For the stochastic integral term, we apply a standard moment inequality for integrals with respect to symmetric $\alpha$-stable processes (see, e.g., \cite[Lemma 1]{Hashimoto}).
This, combined with the uniform boundedness of $\sigma_n$ from Assumption A(iii), yields:
\begin{align*}
    \e\left[\sup_{0\le t\le T}\left|\int_0^t \sigma_n\left(X_{s-}^{(n)}\right)\rd Z_s\right|^p\right]
    \le
    C_{p,\alpha}\e\left[\left(\int_0^T \left|\sigma_n\left(X_{s-}^{(n)}\right)\right|^\alpha\rd s\right)^{p/\alpha}\right]
    \le
    C_{p,\alpha} \sup_{n\in\n}\left\|\sigma_n\right\|^p_\infty T^{p/\alpha}
\end{align*}
The constant $C_{p,\alpha}$ depends only on $p$ and $\alpha$.

Since each of the three terms is bounded by a constant independent of $n$, we conclude that the sequence $\left(\sup_{t\in[0,T]}|X_t^{(n)}|\right)_{n\in\n}$ is uniformly bounded in $L^p$.
This implies uniform integrability.
\end{proof}

We now establish the convergence of the sequence of solutions $\left(X^{(n)}\right)_{n\in\n}$ under Assumption A.
First, let us consider the limits of the coefficients.
Since $(b_n)$ and $(\sigma_n)$ are Cauchy sequences in their respective weighted norms and also satisfy the uniform Lipschitz/H\"older conditions \eqref{Uni_Lip_b_n} and \eqref{Uni_hol_sigma_n}, they converge pointwise to some limit functions $b_\infty$ and $\sigma_\infty$.
By a standard argument (Arzel\'a–Ascoli theorem), this convergence is uniform on compact sets.
The limiting functions $b_\infty$ and $\sigma_\infty$ inherit the same uniform bounds and continuity properties:
\begin{align*}
    \sup_{x\in\real}|b_\infty(x)| < \infty, \quad &|b_\infty(x)-b_\infty(y)| \le K|x-y|, \\
    0 < \inf_{x\in\real}\sigma_\infty(x), \quad \sup_{x\in\real}\sigma_\infty(x) < \infty, \quad &|\sigma_\infty(x)-\sigma_\infty(y)| \le K|x-y|^{\eta}.
\end{align*}
Next, we show that the sequence of solutions $\left(X^{(n)}\right)_{n\in\n}$ converges in $L^p$ for $p \in (1, \alpha)$.
Specifically, we will show that it is a Cauchy sequence in the space $L^p(\Omega; D([0,T]))$, equipped with the supremum norm.
To do this, we need to show that $\\e\left[\sup_{t\in[0,T]}\left|X_t^{(n)} - X_t^{(m)}\right|^p\right] \to 0$ as $n,m \to \infty$.
Our stability estimate in Theorem \ref{mainthm2} provides a bound on the tail probability $\p\left(\sup_{t\in[0,T]}\left|X_t^{(n)} - X_t^{(m)}\right|^{\alpha-1} > h\right)$.
Since the coefficient sequences $\left(b_n\right)_{n\in\n}$ and $\left(\sigma_n\right)_{n\in\n}$ are Cauchy in their respective norms, the right-hand side of the inequality in Theorem \ref{mainthm2} converges to zero as $n,m \to \infty$ for any fixed $h>0$.
This implies that the sequence $\left(\sup_{t\in[0,T]}|X_t^{(n)} - X_t^{(m)}|\right)_{n,m\in\n}$ converges to zero in probability.

To upgrade this convergence in probability to convergence in $L^p$, we need uniform integrability.
Lemma \ref{Uni_integ_Xn} establishes that the sequence $\left(\sup_{t\in[0,T]} \left|X_t^{(n)}\right|\right)_{n\in\n}$ is uniformly bounded in $L^p$.
This implies that the sequence of differences, $\left(\sup_{t\in[0,T]} |X_t^{(n)} - X_t^{(m)}|\right)_{n,m\in\n}$, is also uniformly integrable.

Therefore, the convergence in probability, combined with uniform integrability, guarantees that $\left(X^{(n)}\right)_{n\in\n}$ is a Cauchy sequence in $L^p(\Omega; D([0,T]))$.
Thus, there exists a limiting process $X^\infty$ to which $X^{(n)}$ converges in $L^p$.

Applying Theorem \ref{mainthm2}, we can find a further subsequence $(n_k)_{k\in\n}$ such that the following limit $\displaystyle{X^\infty= \lim_{k\rightarrow\infty}X^{(n_k)}}$ exists almost surely.
The convergence of this subsequence to a limit process $X^\infty$ holds almost surely in the space of c\`adl\`ag functions equipped with the supremum norm, which implies the limit path is well-defined for all $t\in[0,T]$.
The limit $X^\infty$ is c\`adl\`ag since $X^{(n_k)}$ converges in supremum norm by Theorem \ref{mainthm2} (see \cite[P.140]{Applebaum}).
We prove that the limit $X^\infty$ is a solution of SDE \eqref{SDE}.
\begin{Cor}\label{ConstructLemma}
Suppose that $\left(x_0^{(n)}\right)_{n\in\n}$, $(b_n)_{n\in\n}$ and $(\sigma_n)_{n\in\n}$ satisfies Assumption A.
We define $V=\left(V_t\right)_{0 \le t \le T}$ as
\begin{align*}
V_t=x_0+\int_0^t b_\infty\left(X_{s-}^\infty\right)\rd s+ \int_0^t\sigma_{\infty}\left(X_{s-}^\infty\right)\rd Z_s,
\end{align*}
Then, $$\p\left(V_t=X_t^\infty,\textrm{ for each } t\in[0,T]\right)=1.$$
\end{Cor}

\begin{proof}
The existence of the limit of subsequences of solutions $\left(X^{(n)},b_n,\sigma_n\right)_{n\in\n}$ was discussed above.
From Theorem \ref{mainthm} and the bounded convergence theorem, we have
\begin{align*}
\e\left[\left|V_t-X_t^\infty\right|^{\alpha-1}\right]=\lim_{k\rightarrow\infty}\e\left[\left|V_t-X_t^{(n)}\right|^{\alpha-1}\right].
\end{align*}
In the same way as shown in the proof of Theorem \ref{mainthm}, we have
\begin{align}
\left|V_t-X_t^{(n)}\right|^{\alpha-1}\le \left|x_0-x_0^{(n)}\right|^{\alpha-1} + \eps^{\alpha-1}+ \hat{M}_t^{\delta,\eps,n}+ \hat{I}_t^{\delta,\eps,b,n}+ \hat{I}_t^{\delta,\eps,\sigma,n}, \label{construct1}
\end{align}
where 
\begin{align*}
\hat{M}_t^{\delta,\eps,n}&\coloneqq \int_{0}^{t}\int_{\real \setminus \{0\}}\left\{ \ude\left(V_{s-}-X_{s-}^{(n)}+\left(\sigma_\infty\left(X_s^\infty\right)-\sigma_{n_k}\left(X_s^{(n)}\right)\right)z\right)-\ude\left(V_{s-}-X_{s-}^{(n)}\right)\right\}\tilde{N}(\rd z,\rd s),\\
\hat{I}_t^{\delta,\eps,b,n}&\coloneqq \int_0^t \left|\ude'\left(V_s-X_s^{(n)}\right)\right|\left|b_\infty\left(X_s^\infty\right)-b_{n_k}\left(X_s^{(n)}\right)\right|\rd s,\\
\hat{I}_t^{\delta,\eps,\sigma,n}&\coloneqq \frac{2 C_{\alpha}}{\log \delta}\left(\frac{\delta}{\eps}\right) \int_0^t \left|\sigma_\infty(X_s^\infty)-\sigma_{n_k}\left(X_s^{(n)}\right)\right|^\alpha \rd s.
\end{align*}
Following the same arguments as in Subsection \ref{M-martingale}, we can show that $\hat{M}^{\delta,\eps,n}\coloneqq \left(\hat{M}_t^{\delta,\eps,n}\right)_{0\le t \le T}$ is a martingale.
First, we show the case where the coefficient satisfies assumption A.

By using Fatou's lemma and $\left(X_t^{(n)}\right)_{k\in\n}$ converges to $X^\infty$ in $L^p$ for any $p\in(1,\alpha)$, we obtain
\begin{align*}
    \sup_{0\le t\le T}\e\left[\left|X_t^{\infty}\right|^p\right]
    \le 
    \sup_{0\le t\le T}\liminf_{n\to\infty}\e\left[\left|X_t^{(n)}\right|^p\right]<\infty.
\end{align*}
Hence, $\left|X_t^\infty-X_t^{(n)}\right|$ is $L^p$ integrable for each $t\in[0,T]$.

By the inequality $|x+y|^p\le 2^{p-1}\left(|x|^p+|y|^p\right)$ for any $x,y\in\real$ and any $p\ge 1$, we have
\begin{align*}
\hat{I}_t^{\delta,\eps,b,n}&\le \|\ude'\|_\infty \int_0^t \left|b_\infty\left(X_s^\infty\right)-b_\infty\left(X_s^{(n)}\right)\right|\rd s + \|\ude'\|_\infty \int_0^t \left|b_\infty\left(X_s^{(n)}\right)-b_n\left(X_s^{(n)}\right)\right|\rd s,\\
\hat{I}_t^{\delta,\eps,\sigma,n}&\le \frac{2^\alpha C_{\alpha}}{\log \delta}\left(\frac{\delta}{\eps}\right)\int_0^t \left|\sigma_\infty(X_s^\infty)-\sigma_\infty\left(X_s^{(n)}\right)\right|^\alpha \rd s + \frac{2^\alpha C_{\alpha}}{\log \delta}\left(\frac{\delta}{\eps}\right)\int_0^t \left|\sigma_\infty\left(X_s^{(n)}\right)-\sigma_n\left(X_s^{(n)}\right)\right|^\alpha\rd s.
\end{align*}
Note that any bounded, $q$-H\"older continuous function is also $r$-H\"older continuous for any exponent $r<q$.
By using the Vitali convergence theorem, we have
\begin{align*}
\lim_{n\to\infty} \int_0^T \e\left[\left|b_\infty\left(X_s^\infty\right)-b_\infty\left(X_s^{(n)}\right)\right|\right]\rd s
=
\lim_{n\to\infty} \int_0^T \e\left[\left|\sigma_\infty\left(X_s^\infty\right)-\sigma_\infty\left(X_s^{(n)}\right)\right|^\alpha\right]\rd s
=
0.
\end{align*}
Here, by the assumption A on $b_n$ and $\sigma_n$, it follows that each process $X^{(n)}=\left(X_t^{(n)}\right)_{t\in[0,T]}$ has a transition density function for each $t\in(0,T]$ and each $n\in\n$.
The density of $X^{(n)}$ has an upper bound by \eqref{SDE_density_estimate}, so that
we have by Fubini's theorem, 
\begin{align}
\lim_{n\to\infty} \int_0^T \e\left[\left|b_\infty\left(X_s^{(n)}\right)-b_{n_k}\left(X_s^{(n)}\right)\right|\right]\rd s
=
\lim_{n\to\infty} \int_0^T \e\left[\left|\sigma_\infty\left(X_s^{(n)}\right)-\sigma_{n_k}\left(X_s^{(n)}\right)\right|^\alpha\right]\rd s
=
0.\label{conv.integral_by_density}
\end{align}
Taking the expectation in \eqref{construct1} and letting $n\to\infty$, the convergence results established above for both parts of the integral terms imply that $\e\left[\hat{I}_t^{\delta,\eps,b,n}\right]$ and $\e\left[\hat{I}_t^{\delta,\eps,\sigma,n}\right]$ converge to zero.
Hence, for any $\eps>0$, we have
\begin{align*}
\lim_{n\rightarrow\infty}\sup_{0\le t\le T}\e\left[\left|V_t-X_t^{(n)}\right|^{\alpha-1}\right]\le  \eps^{\alpha-1}.
\end{align*}
Therefore, we get
\begin{align*}
\lim_{n\rightarrow\infty}\sup_{0\le t\le T}\e\left[\left|V_t-X_t^{(n)}\right|^{\alpha-1}\right]=0.
\end{align*}
Thus, we have
\begin{align*}
\p\left(V_t=X_t^\infty\right)=1 \textrm{ for each } t\in[0,T].
\end{align*}
Here, since each sample path of $V$ and $X^\infty$ is c\`adl\`ag, we have (see \cite{Philip}, Section I, Theorem 2)
\begin{align*}
\p\left(V_t=X_t^\infty,\textrm{ for each } t\in[0,T]\right)=1.
\end{align*}
\end{proof}

\begin{remark}
When $(b_n)_{n\in\n}$ and $(\sigma_n)_{n\in\n}:[0,\infty)\times \real \to \real$ are a Cauchy sequence in the sense of the sup norm, the following assumptions satisfy the same result as Lemma \ref{ConstructLemma}.
There exist nonnegative functions $f_b$ and $f_\sigma$ and a constant $\eta\in[1/\alpha,1]$ such that for any $t\in[0,T]$ and $x,y\in\real$
\begin{align*}
\sup_{n\in\n}\|b_n\|_\infty, \sup_{n\in\n}\|\sigma_n\|_\infty, \int_0^T f_{b}(s) \rd s,\int_0^T f_{\sigma}^2(s) \rd s<\infty,\\
\sup_{n\in\n} \left|b_n(t,x)-b_n(t,y)\right|\le f_{b}(t)|x-y|,\quad
\sup_{n\in\n}\left|\sigma_{n}(t,x)-\sigma_{n}(t,y)\right|\leq f_{\sigma}(t)|x-y|^{\eta},
\end{align*}
When the assumption holds, it can be shown that $\left(\sup_{t\in[0,T]}\left|X^{(n)}_t\right|\right)_{n\in\n}$ is $L^p(\Omega)$ integrable for $p\in(1,\alpha)$ and each integral in \eqref{conv.integral_by_density} converges to zero by evaluating it using the sup-norm of the difference of the coefficients.
The remaining proofs are identical to Corollary \ref{ConstructLemma}.
\end{remark}


\end{document}